\documentclass[11pt,reqno,final]{amsart}
\usepackage{amssymb}
\usepackage{amsfonts}
\usepackage{amsmath}

\newtheorem{theorem}{Theorem}
\newtheorem{lemma}{Lemma}

\newtheorem{corollary}{Corollary}
\newtheorem{definition}{Definition}
\newtheorem{remark}{Remark}
\theoremstyle{remark}
\numberwithin{equation}{section}
\input{tcilatex}

\begin{document}
\title[PARTIAL NULL-CONTROLLABILITY OF EVOLUTION EQUATIONS]{\textbf{PARTIAL
NULL-CONTROLLABILITY OF EVOLUTION EQUATIONS BY ONE-DIMENSIONAL CONTROL}}
\author{B. Shklyar}
\address{Department of Exact Sciences, Holon Academic Institute of
Technology, Holon 58102, Israel}
\email{shk\_b@hit.ac.il}

\begin{abstract}
The problem of partial null controllability for linear autonomous evolution
equations, which are controlled by a one-dimensional control, is under
consideration. The partial null-controllability conditions for coupled
abstract evolution systems have been obtained using the moment problem
approach.
\end{abstract}

\maketitle
\tableofcontents


\textbf{AMS(MOS) Subject Classification}: 93B05, 93B60

\section{Introduction and problem statement}

Partial controllability is the ability to exert control or influence over
only a subset of components or entities within a distributed system, while
leaving others unaffected. This can be useful in scenarios where certain
components need to be modified or managed independently, while others remain
unchanged.

For example, in a distributed database system, partial controllability may
involve the ability to reach data in specific partitions or shards, while
leaving other untouched. In a robotic system, partial controllability may
involve the ability to control the motion of a subset of robots while
allowing others to operate autonomously based on their local sensors and
algorithms.

The controllability conditions for first-order differential control systems
are well-studied. However, high-order differential equations, partial
hyperbolic equations, delay equations, and many others can be written as a
first-order differential equation only by enlarging the dimension of the
state space. In many cases, it is necessary to modify only components of the
original state space, without affecting the rest of state space. This is why
the concept of partial controllability is important.

Let $X,Y,U$ be separable Hilbert spaces. Consider the evolution equation%
\begin{equation}
\dot{x}(t)=Ax(t)+bu\left( t\right) ,  \tag{e1}  \label{e1}
\end{equation}%
\begin{equation}
x(0)=x_{0},  \tag{e2}  \label{e2}
\end{equation}%
\begin{equation}
y(t)=Px(t),t\geq 0,  \tag{e3}  \label{e3}
\end{equation}%
where $x(t)\in X$ is a current state, $x_{0}\in X$ is an initial state; the
operator $A$ generates a strongly continuous $C_{0}$-semigroup $S(t)$ of
bounded operators \cite{HillePhillps}, $b\in X, P:X\rightarrowtail Y$ is a
linear bounded operator The input function $u\left( t\right) $ is assumed to
be locally \ $L^{2}\,$-integrated.

Following \cite{AmmarKhodjaChoulyDuprez} we assume, that the observation
operator $P:X\mathbb{\rightarrowtail }Y$ is an orthogonal projection
operator,

Denote by \ $x\left( t,x_{0},u\left( \cdot \right) \right) $ the mild
solution of equation (\ref{e1}) with initial condition (\ref{e2}).

\begin{definition}
Equation (\ref{e1})--(\ref{e3}) $\ $ is said to be completely
null-controllable on $[0,t_{1}]$, if for each $x_{0}\in X$ there is a
control $u(t),u(\cdot )\in L([0,t_{1}],U)$, such that the corresponding
solution $x\left( t,x_{0},u\left( \cdot \right) \right) $ satisfies the
condition \ 
\begin{equation*}
x\left( t_{1},x_{0},u\left( \cdot \right) \right) =0.
\end{equation*}
\end{definition}

\begin{definition}
\label{D2.1}Equation (\ref{e1})--(\ref{e3}) $\ $ is said to be partially
null-controllable on $[0,t_{1}]$, if for each $x_{0}\in X$ there is a
control $u(t),u(\cdot )\in L([0,t_{1}],U)$, such that the corresponding
solution $x\left( t,x_{0},u\left( \cdot \right) \right) $ satisfies the
condition \ 
\begin{equation}
Px\left( t_{1},x_{0},u\left( \cdot \right) \right) =0.  \tag{e5.1}
\label{e5.1}
\end{equation}
\end{definition}

Further for the sake of brevity we use $x\left( t\right) \ $instead of $%
x\left( t,x_{0},u\left( \cdot \right) \right) .$

The partial controllability has been well investigated for
finite-dimensional systems (see, for example, \cite{GabasovKirillova} and
references therein). However, as far as we know there are very few
investigations on the partial controllability of infinite-dimensional
equations (see \cite{BashirovMahmudovSemiEtikan}, \cite%
{AmmarKhodjaChoulyDuprez} and reference therein).

In recent years, there has been an increasing interest in the study of the
controllability of coupled PDEs. Partial controllability of such systems is
one of important applications to control of these systems. In \cite%
{AmmarKhodjaChoulyDuprez} partial controllability conditions have been
obtained for coupled parabolic systems (see also references therein).
Controllability conditions of \cite{AvdoninRiveroTeresa} (see also
references therein) may be considered as partially null controllability \
conditions for coupled hyperbolic (namely waved) equations, where only
initial condition conditions $x_{0}\in Y$ are considered.

The moment method is a very powerful tool for investigation of (complete)
controllability for linear evolution systems (see \cite{AvdoninIvanov}, \cite%
{AmmarKhodja}, \cite{AvdoninRiveroTeresa} and references therein).

The main goal of this paper is to obtain results related to partial
controllability issues for linear evolution equations using the moment
method. The paper aims to provide insights into controlling such systems and
may discuss aspects related to hyperbolic and parabolic systems. The moment
method is a mathematical technique employed to analyze and establish
controllability in these systems. 


We assume that $A$ has the properties:

(I) the operator $A$ has a purely point spectrum $\sigma _{A}$ which is
either finite or has no finite limit points, and each $\lambda \in \sigma
_{A}$ has a finite multiplicity; %

(II) the family of generalized eigenvectors of the operator $A$ forms a
Riesz basic of $X.$


It is well-known, that $x\left( t\right) $ is defined by the formula 
\footnote{%
The integral is understood in the Bochner sense.}:

\begin{equation}
x\left( t\right) =S(t)x_{0}+\int_{0}^{t}S(t-\tau )bu(\tau )d\tau . 
\tag{eqsol}  \label{eqsol}
\end{equation}



%
%

\section{Coupled evolution systems}




%

Let $X=Y\times Z,$ where $Y$ and $Z$ are separable Hilbert spaces\footnote{%
We will not distinguish between spaces $X$ and $X\times \{0\}, Y$ and $%
\{0\}\times Y$.}. Consider the linear evolution system 
\begin{equation}
\dot{y}\left( t\right) =A_{Y}y\left( t\right) +A_{YZ}z\left( t\right)
+b_{Y}u\left( t\right) ,  \tag{CS1 }  \label{CS1}
\end{equation}%
\begin{equation}
\dot{z}\left( t\right) \ =~~~~~~~~~A_{Z}z\left( t\right) ~+b_{Z}u\left(
t\right) .  \tag{CS2}  \label{CS2}
\end{equation}%
\begin{equation}
P=\left( 
\begin{array}{cc}
I_{Y} & 0 \\ 
0 & 0%
\end{array}%
\right) ,  \tag{P}  \label{P}
\end{equation}%
where the operators $A_{Y}$ and $A_{Z}$ generate strongly continuous
semigroups $S_{Y}\left( t\right) ~$and $S_{Z}\left( t\right) $ in $X$ and Y
correspondingly, and $S_{YZ}\left( t\right) :Z\rightarrowtail Y$ is a
bounded strongly continuous operator.

This system is a partial case of system (\ref{e1})--(\ref{e3}), where $%
x=\left( 
\begin{array}{c}
y \\ 
z%
\end{array}%
\right) ,y\in Y,z\in Z,~A=\left( 
\begin{array}{cc}
A_{Y}, & A_{YZ} \\ 
0 & A_{Z}%
\end{array}%
\right) ,\ A^{\ast }=\left( 
\begin{array}{cc}
A_{Y}^{\ast }, & 0 \\ 
A_{YZ}^{\ast } & A_{Z}^{\ast }%
\end{array}%
\right) ,$ $b=\left( 
\begin{array}{c}
b_{Y} \\ 
b_{Z}%
\end{array}%
\right) .$

We assume, that eigenvalues of the operators $A_{Y}$ and $A_{Z}$ are simple,
and the operator

$A=\left( 
\begin{array}{cc}
A_{Y}, & A_{YZ} \\ 
0 & A_{Z}%
\end{array}%
\right) $ satisfies the conditions (I)--(II).

The linear bounded operator $A_{YZ}$ will be called a couple operator.

We have 
\begin{equation*}
S\left( t\right) =\left( 
\begin{array}{cc}
S_{Y}\left( t\right) , & S_{YZ}\left( t\right) \\ 
0, & S_{Z}\left( t\right)%
\end{array}%
\right) ,
\end{equation*}%
where $S_{Y}\left( t\right) ~$and $S_{Z}\left( t\right) $ are strongly
continuous semigroups, generated by the operators $A_{Y}$ and $A_{Z}$
correspondingly, and $S_{YZ}\left( t\right) :Z\rightarrowtail Y$ is a
bounded strongly continuous operator.

Denote by $\sigma _{Y}$ the spectrum of the operator $A_{Y},$ and by $\sigma
_{Z}$ the spectrum of the operator $A_{Z}.$

Denote also: $I_{Y}:=\left\{ i\in \mathbb{N}:\lambda _{i}\in \sigma
_{Y}\right\} ,I_{Z}:=\left\{ i\in \mathbb{N}:\lambda _{i}\notin \sigma
_{Y}\right\} .$ Obviously, $\sigma _{Y}\cup \sigma _{Z}=\mathbb{\sigma }%
_{A},~I_{Y}\cup I_{Z}=\mathbb{N}.$

Denote $\sigma _{YZ}=\sigma _{Y}\cap \sigma _{Z}.$

\section{The case $\protect\sigma _{YZ}=\varnothing $ \ \ \ \ }

\subsection{Eigenvectors and eigenvalues}

In the beginning consider the case $\sigma _{YZ}=\varnothing $.~{\small %
Obviously},$~I_{Y}\cap I_{Z}=\varnothing $~{\small also}$.$

\subsubsection{Eigenvectors of operator $A$}

Let $\varphi _{i}=\left( 
\begin{array}{c}
\varphi _{Yi} \\ 
\varphi _{Yi}%
\end{array}%
\right) ,i\in \mathbb{N}$ be the eigenvector of the operator $A,$
corresponding to its eigenvalue $\lambda _{i},i\in \mathbb{N}.$ We have

$A\varphi _{i}=\left( 
\begin{array}{c}
A_{Y}\varphi _{Yi}+A_{YZ}\varphi _{Zi} \\ 
A_{Z}\varphi _{Zi}%
\end{array}%
\right) =\lambda _{i}\left( 
\begin{array}{c}
\varphi _{Yi} \\ 
\varphi _{Yi}%
\end{array}%
\right) ,i\in \mathbb{N}.$

Hence, 
\begin{eqnarray*}
A_{Y}\varphi _{Yi}+A_{YZ}\varphi _{Zi} &=&\lambda _{i}\varphi _{Yi},i\in 
\mathbb{N}, \\
A_{Z}\varphi _{Zi} &=&\lambda _{i}\psi _{Zi},i\in \mathbb{N}.
\end{eqnarray*}

1) Let ~$j\in I_{Y}\iff \lambda _{j}\in \sigma _{Y},~$ and let $\varphi
_{Yj} $ be a corresponding eigenvector of the operator $A_{Y}.$ In this case
the vector $\varphi _{j}=\left( 
\begin{array}{c}
\varphi _{Yj} \\ 
0%
\end{array}%
\right) $ is an eigenvector of the operator $A.$

2) Let $k\in I_{Z}\iff \lambda _{k}\in \sigma _{Z},$~and~let $\varphi
_{Zk}\neq 0~$be a corresponding eigenvector of the operator $A_{Z}.$ In this
case the vector $\varphi _{k}=\left( 
\begin{array}{c}
\varphi _{Yk} \\ 
\varphi _{Zk}%
\end{array}%
\right) ,~$where$~$the vector $\varphi _{Yk}~$\ \ is a solution of the
operator equation 
\begin{equation}
A_{Y}\varphi _{Yk}+A_{YZ}\varphi _{Zk}=\lambda _{k}\varphi _{Yk},\varphi
_{Zk}\neq 0,k\in I_{\mathbb{Z}},  \tag{AphiYk}  \label{AphiYk}
\end{equation}%
is an eigenvector of the operator $A.$

\begin{remark}
Since $\lambda _{k}\ $is a regular value of the operator $A_{Y},$ the
operator equation (\ref{AphiYk}) is always solvable. Its solution is: 
\begin{equation}
\varphi _{Yk}=\left( A_{Y}-\lambda _{k}I_{Y}\right) ^{-1}A_{YZ}\varphi
_{Zk},k\in I_{\mathbb{Z}}.  \tag{phiYk}  \label{phiYk}
\end{equation}
\end{remark}

\subsubsection{Eigenvectors of adjoint operator $A^{\ast }$}

Let $\psi _{i}=\left( 
\begin{array}{c}
\psi _{Yi} \\ 
\psi _{Zi}%
\end{array}
\right),i\in \mathbb{N}$ be the eigenvector of the operator $A^{\ast },$
corresponding to its eigenvalue $\bar{\lambda}_{i},i\in \mathbb{N}.$ We have

$A^{\ast }\psi _{i}=\left( 
\begin{array}{c}
A_{Y}^{\ast }\psi _{Yi} \\ 
A_{YZ}^{\ast }\psi _{Yi}+A_{Z}^{\ast }\psi _{Zi}%
\end{array}%
\right) =\bar{\lambda}_{i}\left( 
\begin{array}{c}
\psi _{Yi} \\ 
\psi _{Zi}%
\end{array}%
\right) ,i\in \mathbb{N}.$

Hence, 
\begin{eqnarray*}
A_{Y}^{\ast }\psi _{Yi} &=&\bar{\lambda}_{i}\psi _{Yi},i\in \mathbb{N}, \\
A_{YZ}^{\ast }\psi _{Yi}+A_{Z}^{\ast }\psi _{Zi} &=&\bar{\lambda}_{i}\psi
_{Zi},i\in \mathbb{N}.
\end{eqnarray*}

1) Let ~$k\in I_{Z}\iff \lambda _{k}\in \sigma _{Z},$ and let $\psi _{Zk}$
be a corresponding eigenvector of the operator $A_{Z}^{\ast }.$ In this case
the vector $\psi _{k}=\left( 
\begin{array}{c}
0 \\ 
\psi _{Zk}%
\end{array}%
\right) $ is an eigenvector of the operator $A^{\ast }.$

2) Let~$j\in I_{Y}\iff \lambda _{j}\in \sigma _{Y},$ and let $\psi _{Zj}$ be
a corresponding eigenvector of the operator $A_{Y}^{\ast }.$ In this case
the vector $\psi _{j}=\left( 
\begin{array}{c}
\psi _{Yj} \\ 
\psi _{Zj}%
\end{array}%
\right) ,$ where the vector $\psi _{Zj}~$\ \ is a solution of the equation 
\begin{equation}
A_{YZ}^{\ast }\psi _{Yj}+A_{Z}^{\ast }\psi _{Zj}=\bar{\lambda}_{j}\psi
_{Zj},j\in I_{Y},  \tag{A*psiZj}  \label{A*psiZj}
\end{equation}%
is an eigenvector of the operator $A^{\ast }.$

\begin{remark}
Since $\bar{\lambda}_{j}\ $is a regular value of the operator $A_{Z}^{\ast
}, $ the operator equation (\ref{A*psiZj}) is always solvable. Its solution
is: 
\begin{equation}
\psi _{Zj}=\left( A_{Z}^{\ast }-\bar{\lambda}_{k}I_{Z}\right)
^{-1}A_{YZ}^{\ast }\psi _{Yj},0,k\in I_{\mathbb{Z}}.  \tag{psiZj}
\label{psiZj}
\end{equation}
\end{remark}

Denote: $b_{i}=\left( b,\psi _{i}\right) ,~i\in \mathbb{N}.$

We have 
\begin{equation*}
b_{i}=\left( 
\begin{array}{c}
b_{Yi} \\ 
b_{Zi}%
\end{array}%
\right) ,i\in \mathbb{N}.
\end{equation*}%
where 
\begin{equation}
b_{Yi}=\left( b_{Y},\psi _{Yi}\right) ,~b_{Zi}=\left( b_{Z},\psi
_{Zi}\right) ,i\in \mathbb{N}.  \tag{bi}  \label{bi}
\end{equation}%
If $k\in I_{Z},$ then $\psi _{Yk}=0,$ so $b_{k}=\left( 
\begin{array}{c}
0 \\ 
b_{Zk}%
\end{array}%
\right) .$

The eigenvectors $\varphi _{i},\psi _{i},i\in \mathbb{N}$ of operators $A~$%
and $A^{\ast }$ are chosen such that following orthogonality conditions hold:

\begin{equation*}
\left( \varphi _{m},\psi _{n}\right) =\delta _{mn},~m,n\in \mathbb{N}.
\end{equation*}%
Hence 
\begin{equation}
\left( \varphi _{Yj},\psi _{Yl}\right) = \delta _{jl},j,l\in I_{Y}, 
\tag{ortY}  \label{ortY}
\end{equation}
\begin{equation}
\left( \varphi _{Zk},\psi _{Zm}\right) = \delta _{km},k,m\in I_{Z}, 
\tag{ortZ}  \label{ortZ}
\end{equation}
\begin{equation}
\left( \varphi _{Yk},\psi _{Yj}\right) = -\left( \varphi _{Zk},\psi
_{Zj}\right) ,k\in I_{Z},j\in I_{Y}.  \tag{ortYZ}  \label{ortYZ}
\end{equation}

Denote:

$\Phi _{Y}:=\left\{ \varphi _{j}=\left( 
\begin{array}{c}
\varphi _{Yj} \\ 
0%
\end{array}%
\right) ,j\in I_{Y}\right\} ,\Phi _{Z}:=\left\{ \varphi _{k}=\left( 
\begin{array}{c}
\varphi _{Yk} \\ 
\varphi _{Zk}%
\end{array}%
\right) ,k\in I_{Z}\right\} ,$

$\Psi _{Y}:=\left\{ \psi _{j}=\left( 
\begin{array}{c}
\psi _{Yj} \\ 
\psi _{Zj}%
\end{array}%
\right) ,j\in I_{Y}\right\} ,\Psi _{Z}:=\left\{ \psi _{k}=\left( 
\begin{array}{c}
0 \\ 
\psi _{Zk}%
\end{array}%
\right) ,k\in I_{Z}\right\} $

\begin{lemma}
\label{L1} The sequences $\left\{ \varphi _{Yj},j\in I_{Y}\right\} $ and $%
\left\{ \psi _{Yj},j\in I_{Y}\right\} $ are Riesz basics of the space $Y.$

The sequences $\left\{ \varphi _{Zk},k\in I_{Z}\right\} $ and $\left\{ \psi
_{Zk},k\in I_{Z}\right\} $ are Riesz basics of the space $Z.$
\end{lemma}

\begin{proof}
This follows from results of \cite{Stoeva} Theorem 2.1 (see also \cite{Bari}%
). 
\end{proof}


\section{Partial controllability conditions}

Below all definitions and notations of strongly minimal sequences belong to 
\cite{Taldykin}.

Let 
\begin{equation}
\left\{ x_{1},x_{2},...,x_{n},...\right\}  \tag{s}  \label{s}
\end{equation}
be a minimal sequence of \ elements of a separable Hilbert space $H.$

The infinite-dimensional matrix\ 
\begin{equation}
\left\{ \varphi _{ij}=(x_{i},x_{j}),i,j\in \mathbb{N}\right\}  \tag{GM}
\label{GM}
\end{equation}
is called a Gram matrix of family (\ref{s}).\newline

Let 
\begin{equation}
G_{n}=\left\{ \varphi _{ij}=(x_{i},x_{j}),i,j=1,2,...,n\right\}  \tag{GMn}
\label{GMn}
\end{equation}%
be $n\times n$ cutout of (\ref{GM}), and let $0\leq \lambda _{1}^{n}\leq
\lambda _{2}^{n}\leq ...\leq \lambda _{n}^{n}$ be eigenvalues of (\ref{GMn})%
\footnote{%
Matrix (\ref{GMn}) is a positive Hermitian matrix for any $n\in \mathbb{N}.$}%
. It is well-known, that $\lambda _{1}^{n}\geq 0$ is not increasing, $%
\lambda _{n}^{n}$ is not decreasing, so there exists $\lim\limits_{n%
\rightarrow \infty }\lambda _{1}^{n}\geq 0.$

\begin{definition}
If $\lim\limits_{n\rightarrow \infty }\lambda _{1}^{n}>0,$ the sequence (\ref%
{s}) is called strongly minimal sequence\footnote{%
The definition of strongly minimal sequences is equivalent to the definition
of Riesz-Fisher sequences \cite{CasazzaChristensenLindner}, \cite{Young80}.
In our opinion,the strong minimality concept better explains the
relationship of such sequences with minimal systems.}.
\end{definition}

Denote:

$b_{j}\left( t\right) =b_{j}-e^{\lambda _{j}t}e^{-\lambda _{j}t_{1}}\left(
S_{Z}\left( t_{1}-t\right) b_{Z,},\psi _{Zj}\right) ,j\in I_{Y},~t\in \left[
0,t_{1}\right] .$

\begin{theorem}
\label{T1} Let the \ series 
\begin{equation}
\sum\limits_{j\in \mathbb{I}_{Y}}e^{-2\lambda _{j}t_{1}}\left\vert \left(
S_{Z}\left( t_{1}\right) x_{Z}^{0},\psi _{Zj}\right) \right\vert ^{2} 
\tag{x\_z(t1)}  \label{x_z(t1)}
\end{equation}%
converges for any $x_{Z}^{0}\in Z$.

Equation (\ref{CS1}) is partially null controllable on $[0,t_{1}]$~ if and
only if the family of exponents%
\begin{equation}
\left\{ e^{-\lambda _{j}t}b_{j}\left( t\right) ,j\in I_{Y},t\in \left[
0,t_{1}\right] \right\}  \tag{-exp}  \label{-exp}
\end{equation}%
is strongly minimal.
\end{theorem}


\begin{proof}
Since the family of eigenvectors $\left\{ \varphi _{i},i\in \mathbb{N}%
\right\} ~$of the operator $A~$is a Riesz basic of $X$ and the family of
eigenvectors $\left\{ \psi _{i},i\in \mathbb{N}\right\} ~$of the operator $%
A^{\ast }~$is a Riesz basic of $X~$as well , for any $x^{0}\in X$ we have

\begin{equation}
x^{0}=\sum\limits_{i\in \mathbb{N}}\left( x^{0},\psi _{i}\right) \varphi
_{i}=\sum\limits_{j\in I_{Y}}\left( x^{0},\psi _{j}\right) \left( 
\begin{array}{c}
\varphi _{Yj} \\ 
0%
\end{array}%
\right) +\sum\limits_{k\in \mathbb{I}_{Z}}\left( x^{0},\psi _{k}\right)
\left( 
\begin{array}{c}
\varphi _{Yk} \\ 
\varphi _{Zk}%
\end{array}%
\right) ,  \tag{x0}  \label{x0}
\end{equation}

\begin{equation}
x_{Y}^{0}=\sum\limits_{j\in \mathbb{I}_{Y}}\left( x^{0},\psi _{j}\right)
\varphi _{Yj}+\sum\limits_{k\in \mathbb{I}_{Z}}\left( x^{0},\psi _{k}\right)
\varphi _{Yk},  \tag{xY0}  \label{xY0}
\end{equation}

\bigskip 
\begin{equation}
\left( x_{Y}^{0},\psi _{Yj}\right) =\left( x^{0},\psi _{j}\right) +\left(
\sum\limits_{k\in \mathbb{I}_{Z}}\left( x^{0},\psi _{k}\right) \varphi
_{Yk},\psi _{Yj}\right) =  \tag{xY0psiY}  \label{xY0psiY}
\end{equation}
\begin{equation}
=\left( x_{Y}^{0},\psi _{Yj}\right) +\left( x_{Z}^{0},\psi _{Zj}\right)
+\left( \sum\limits_{k\in \mathbb{I}_{Z}}\left( x^{0},\psi _{k}\right)
\varphi _{Yk},\psi _{Yj}\right) \footnote{$\mathrm{hence~}\left(
x_{Z}^{0},\psi _{Zj}\right) +\left( \sum\limits_{k\in \mathbb{I}_{Z}}\left(
x^{0},\psi _{k}\right) \varphi _{Yk},\psi _{Yj}\right) =0.\mathrm{V}$}. 
\notag
\end{equation}%
After substituting $S\left( t_{1}\right) x^{0}$ instead of $x^{0}$ in (\ref%
{x0}), we obtain

\begin{equation*}
x^{0}\left( t_{1}\right) :=S\left( t_{1}\right) x^{0}=\sum\limits_{i\in 
\mathbb{N}}\left( S\left( t_{1}\right) x^{0},\psi _{i}\right) \varphi _{i}=
\end{equation*}

\begin{equation}
=\sum\limits_{j\in I_{Y}}\left( x^{0},\psi _{j}\right) e^{\lambda
_{j}t_{1}}\left( 
\begin{array}{c}
\varphi _{Yj} \\ 
0%
\end{array}%
\right) +\sum\limits_{k\in I_{Z}}\left( x^{0},\psi _{k}\right) e^{\lambda
_{k}t_{1}}\left( 
\begin{array}{c}
\varphi _{Yk} \\ 
\varphi _{Zk}%
\end{array}%
\right) ,  \tag{S(t1)x0}  \label{S(t1)x0}
\end{equation}

Using $b$ instead of $x^{0}$ and $t_{1}-\tau $ instead of $t_{1}$ in (\ref%
{S(t1)x0}) we obtain after integration, that

%
\begin{equation*}
\int_{0}^{t_{1}}S\left( t_{1}-\tau \right) bu\left( \tau \right) d\tau =
\end{equation*}%
\begin{equation}
=\int_{0}^{t_{1}}\left[ \sum\limits_{j\in \mathbb{I}_{Y}}b_{j}e^{\lambda
_{j}\ (t_{1}-\tau )\ }\left( 
\begin{array}{c}
\varphi _{Yj} \\ 
0%
\end{array}%
\right) +\sum\limits_{k\in I_{Z}}b_{k}e^{\lambda _{k}\left( t_{1}-\tau
\right) }\left( 
\begin{array}{c}
\varphi _{Yk} \\ 
\varphi _{Zk}%
\end{array}%
\right) \right] u\left( \tau \right) d\tau .  \tag{S(t1)bu(t)}
\label{S(t1)bu(t)}
\end{equation}%
Substituting \ (\ref{S(t1)x0}) and (\ref{S(t1)bu(t)}) into (\ref{eqsol}), we
obtain

\begin{equation}
x\left( t_{1}\right) =\sum\limits_{j\in I_{Y}}\left( \left( x^{0},\psi
_{j}\right) e^{\lambda _{j}t_{1}}+\left( \int_{0}^{t_{1}}b_{j}e^{\lambda
_{j}\left( \ t_{1}-\tau \right) \ }u\left( \tau \right) d\tau \right)
\right) \left( 
\begin{array}{c}
\varphi _{Yj} \\ 
0%
\end{array}%
\right) +  \tag{x(t1)}  \label{x(t1)}
\end{equation}%
\begin{equation*}
+\sum\limits_{k\in I_{Z}}\left( \left( x^{0},\psi _{k}\right) e^{\lambda
_{k}t_{1}}+\left( \int_{0}^{t_{1}}\left( b,\psi _{k}\right) e^{\lambda
_{k}\left( t_{1}-\tau \right) }u\left( \tau \right) d\tau \right) \right)
\left( 
\begin{array}{c}
\varphi _{Yk} \\ 
\varphi _{Zk}%
\end{array}%
\right) ,
\end{equation*}

%
\begin{equation}
x_{Y}\left( t_{1}\right) =\sum\limits_{j\in I_{Y}}\left( x^{0},\psi
_{j}\right) e^{\lambda _{j}t_{1}}\varphi _{Yj}+\sum\limits_{k\in
I_{Z}}\left( x^{0},\psi _{k}\right) e^{\lambda _{k}t_{1}}\varphi _{Yk}+ 
\tag{xY(t1)}  \label{xY(t1)}
\end{equation}%
\begin{equation*}
+\int_{0}^{t_{1}}\sum\limits_{j\in I_{Y}}\left( b_{j}e^{\lambda _{j}\left( \
t_{1}-\tau \right) \ }\varphi _{Yj}+\sum\limits_{k\in I_{Z}}b_{k}e^{\lambda
_{k}\left( t_{1}-\tau \right) }\varphi _{Yk}\right) u\left( \tau \right)
d\tau .
\end{equation*}

Taking into account, that 
\begin{equation*}
\left( x^{0},\psi _{k}\right) =\left( x_{Z}^{0},\psi _{Zk}\right) ~~\mathrm{%
and~}\left( ~\varphi _{Yk},\psi _{Yj}\right) +\left( \varphi _{Zk},\psi
_{Zj}\right) ~=0,\forall j\in I_{Y},k\in I_{Z},
\end{equation*}%
we obtain%
\begin{equation*}
\left( x_{Y}\left( t_{1}\right) ,\psi _{Yj}\right) =
\end{equation*}%
\begin{equation*}
=e^{\lambda _{j}t_{1}}\left( x^{0},\psi _{j}\right) +\left(
\sum\limits_{k\in I_{\mathbb{Z}}}e^{\lambda _{k}t_{1}}\left( x^{0},\psi
_{k}\right) \varphi _{Yk},\psi _{Yj}\right) +
\end{equation*}%
\begin{equation*}
+\left( \int_{0}^{t_{1}}e^{\lambda _{j}\left( t_{1}-\tau \right) \
}b_{j}+\sum\limits_{k\in I_{\mathbb{Z}}}\left( e^{\lambda _{k}\left(
t_{1}-\tau \right) }\left( b_{Z},\psi _{Zk}\right) \varphi _{Yk},\psi
_{Yj}\right) \right) u\left( \tau \right) d\tau =.
\end{equation*}%
\begin{equation*}
=e^{\lambda _{j}t_{1}}\left( x^{0},\psi _{j}\right) -\sum\limits_{k\in I_{%
\mathbb{Z}}}\left( e^{\lambda _{k}t_{1}}\left( x_{Z}^{0},\psi _{Zk}\right)
\varphi _{Zk},\psi _{Zj}\right) +
\end{equation*}%
\begin{equation*}
+\int_{0}^{t_{1}}\left( e^{\lambda _{j}\left( t_{1}-\tau \right) \
}b_{j}-e^{-\lambda _{j}t_{1}}\sum\limits_{k\in I_{\mathbb{Z}}}\left(
e^{\lambda _{k}\left( t_{1}-\tau \right) }\left( b_{Z},\psi _{Zk}\right)
\varphi _{Yk},\psi _{Yj}\right) \right) u\left( \tau \right) d\tau ,~j\in
I_{Y}~.
\end{equation*}%
Since the sequence $\left\{ \varphi _{Zk},k\in I_{Z}\right\} $ is a Riesz
basic of the space $Z$ (see Lemma \ref{L1}), we have

\begin{equation}
S_{Z}\left( t\right) z=\sum\limits_{k\in I_{\mathbb{Z}}}e^{\lambda
_{k}t}\left( z,\psi _{Zk}\right) \varphi _{Zk},\forall t>0,\forall z\in Z~. 
\tag{Sz(t)Xz}  \label{Sz(t)Xz}
\end{equation}%
Hence 
\begin{equation}
\left( x_{Y}\left( t_{1}\right) ,\psi _{Yj}\right) =e^{\lambda _{j}t_{1}} 
\left[ \left( x^{0},\psi _{j}\right) -e^{-\lambda _{j}t_{1}}\left(
S_{Z}\left( t_{1}\right) x_{Z}^{0},\psi _{Zj}\right) \right] +  \tag{xYpsiYj}
\label{xYpsiYj}
\end{equation}%
\begin{equation*}
+\left( \int_{0}^{t_{1}}e^{\lambda _{j}\ t_{1}\ }\left[ e^{-\lambda _{j}\tau
\ }b_{j}-e^{-\lambda _{j}t_{1}}\left( S_{Z}\left( t_{1}-\tau \right)
b_{Z},\psi _{Zj}\right) \right] \right) u\left( \tau \right) d\tau ,~j\in
I_{Y}~
\end{equation*}%
$.$

%
%

Thus, the partial null-controllability we are proving is equivalent \ to the
existence of a solution $u\left( \cdot \right) \in L_{2}\left[ 0,t_{1}\right]
$ to the moment problem%
\begin{equation*}
-\left( x^{0},\psi _{j}\right) +e^{-\lambda _{j}t_{1}}\left( S_{Z}\left(
t_{1}\right) x_{Z,}^{0},\psi _{Zj}\right) =
\end{equation*}%
\begin{equation}
=\int_{0}^{t_{1}}\left[ e^{-\lambda _{j}\tau \ }b_{j}-e^{-\lambda
_{j}t_{1}}\left( S_{Z}\left( t_{1}-\tau \right) b_{Z},\psi _{Zj}\right) %
\right] u\left( t\right) dt,  \tag{ LMP }  \label{LMP}
\end{equation}%
$\forall j\in I_{Y},~~\forall x^{0}\in X.$

By convergence of series (\ref{x_z(t1)}) we obtain~ 
\begin{equation*}
\sum\limits_{j\in I_{Y}}\left\vert \left( x^{0},\psi _{j}\right)
-e^{-\lambda _{j}t_{1}}\left( S_{Z}\left( t_{1}\right) x_{Z,}^{0},\psi
_{Zj}\right) \right\vert ^{2}<+\infty ,\forall x^{0}\in X.
\end{equation*}%
Hence for any $\forall x^{0}\in X$ moment problem (\ref{LMP}) admits a
solution $u\left( \cdot \right) \in L_{2}\left[ 0,t_{1}\right] $ if and only
if the family (\ref{-exp}) is strongly minimal.

This proves the theorem.
\end{proof}

\begin{remark}
Of course series (\ref{x_z(t1)}) converges, if $I_Y$ is finite. If $I_Y$ is
infinite, then the convergence of series (\ref{x_z(t1)}) essentially depends
on the structure of operator $A_{YZ}$.
\end{remark}

\textbf{Example. }Let $X=L_{2}\left[ 0,\pi \right] \times \mathbb{R},~Y=L_{2}%
\left[ 0,\pi \right] ,~Z=\mathbb{R},A_{Y}=\Delta ~$with boundary conditions $%
y\left( 0\right) =y\left( \pi \right) =0,\ A_{Z}=w\in \mathbb{R},~0\leq
w<1,~A_{YZ}=a\in Y;$

$I_{Y}=\mathbb{N}\backslash \{1\}=2,3,...,~I_{Z}=\{1\};~\lambda
_{1}=w,\varphi _{Z1}=\psi _{Z1}=1,$

$\lambda _{l}=-\left( l-1\right) ^{2},~\varphi _{Yl}=\psi _{Yl}=\sqrt{\frac{2%
}{\pi }}\sin \left( l-1\right) x,~l\in \mathbb{N}\backslash \{1\}~;$ $~$

$A_{YZ}^{\ast }\psi _{Yl}+w\psi _{Zl}=\lambda _{l}\psi _{Zl}\iff $

$\iff \psi _{Zl}=\frac{A_{YZ}^{\ast }\psi _{Yl}}{\lambda _{l}-w}=-\frac{%
\left( a,\psi _{Yl}\right) }{\left( l-1\right) ^{2}+w}=-\frac{\sqrt{\frac{2}{%
\pi }}\int_{0}^{\pi }a\left( x\right) \sin \left( \left( l-1\right) x\right)
dx}{\left( l-1\right) ^{2}+w},l\in \mathbb{N}\backslash \{1\},$

Obviously,$~S_{Z}\left( t_{1}\right) x_{Z,}^{0}=e^{wt_{1}}x_{Z}^{0}.$ Using $%
j=l-1.l=2,3...,$we obtain$~$

$\psi _{Zj+1}=-\sqrt{\frac{2}{\pi }}\frac{\int_{0}^{\pi }a\left( x\right)
\sin jxdx}{j^{2}+w},j\in \mathbb{N}.$

Hence, for any $l\in I_{Y}~~$we have

$(e^{-\lambda _{l}t_{1}}\left( S_{Z}\left( t_{1}\right) x_{Z,}^{0},\psi
_{Zl}\right) =e^{\left( l-1\right) ^{2}t_{1}}e^{wt_{1}}x_{0}^{2}\psi _{Zl}=$

$=-\sqrt{\frac{2}{\pi }}x_{0}^{2}e^{wt_{1}}e^{j^{2}t_{1}}\frac{\int_{0}^{\pi
}a\left( x\right) \sin jxdx}{j^{2}+w},j\in \mathbb{N}.$

Obvious, $\sum\limits_{j=1}^{\infty }\frac{1}{j^{2}}<+\infty ,$ hence by
Riesz-Fisher Theorem there exists $a\left( x\right) \in L_{2}\left[ 0,\pi %
\right] ,$ such that $\int_{0}^{\pi }a\left( x\right) \sin jxdx=\frac{1}{j}%
,~~~~j\in \mathbb{N}.$

Also,~ $\sum\limits_{j=1}^{~\infty }e^{-2j^{2}t_{1}}<+\infty ,$ hence by
Riesz-Fisher Theorem there exists $a\left( x\right) \in L_{2}\left[ 0,\pi %
\right] ,$ such that $\int_{0}^{\pi }a\left( x\right) \sin
jxdx=e^{-j^{2}t_{1}},j\in \mathbb{N}.$

If $\int_{0}^{\pi }a\left( x\right) \sin jxdx=\frac{1}{j},$ then 
\begin{equation*}
\sum\limits_{j\in I_{Y}}\left\vert e^{-\lambda _{j}t_{1}}\left( S_{Z}\left(
t_{1}\right) x_{Z,}^{0},\psi _{Zj}\right) \right\vert ^{2}=\sqrt{\frac{2}{%
\pi }}e^{2wt_{1}}\left\vert x_{Z}^{0}\right\vert
^{2}\sum\limits_{j=1}^{\infty }e^{2j^{2}t_{1}}\frac{1}{j^{2}\left(
j^{2}+w\right) ^{2}}=+\infty .
\end{equation*}

If $\int_{0}^{\pi }a\left( x\right) \sin jxdx=e^{-2j^{2}t_{1}},$then

\begin{eqnarray*}
\sum\limits_{j\in I_{Y}}\left\vert e^{-\lambda _{j}t_{1}}\left( S_{Z}\left(
t_{1}\right) x_{Z,}^{0},\psi _{Zj}\right) \right\vert ^{2} &=&\sqrt{\frac{2}{%
\pi }}\sum\limits_{j=1}^{\infty }e^{2j^{2}t_{1}}\sqrt{\frac{2}{\pi }}%
e^{2wt_{1}}\left\vert x_{Z,}^{0}\right\vert ^{2}\frac{e^{-2j^{2}t_{1}}}{%
\left( j^{2}+w\right) ^{2}}= \\
&=&\sqrt{\frac{2}{\pi }}e^{2wt_{1}}\left\vert x_{Z,}^{0}\right\vert
^{2}\sum\limits_{j=1}^{\infty }\frac{1}{\left( j^{2}+w\right) ^{2}}<+\infty .
\end{eqnarray*}

\subsubsection{The case $b\in Y$}

In this chase problem moment (\ref{LMP}) is simplified.

Let $b_{z}=0,$ i.e. $b=b_{Y}\in Y.$

Controllability of such partial parabolic control systems have important
applications to many fields such as chemistry, physics or biology (see also
Example 3 below).

If $b\in Y,$ then (\ref{CS1})--(\ref{P}) is not completely null-controllable
on $\left[ 0,t_{1}\right] ,$ $\forall t_{1}>0,$ because solutions of
equation (\ref{CS2}) are independent on control. However it may be partially
null controllable.


In this chase problem moment (\ref{LMP}) is simplified:

\begin{equation*}
-\left( x^{0},\psi _{j}\right) +e^{-\lambda _{j}t_{1}}\left( S_{Z}\left(
t_{1}\right) x_{Z,}^{0},\psi _{Zj}\right) =
\end{equation*}%
\begin{equation}
=\int_{0}^{t_{1}}b_{j}e^{-\lambda _{j}\tau \ }u\left( t\right) dt,j\in I_{Y}.
\tag{ LMP0}  \label{LMP0}
\end{equation}

\begin{corollary}
\label{C1}If the \ series (\ref{x_z(t1)}) converges for any $x_{Z}^{0}\in Z,$
then system (\ref{CS1})--(\ref{P}) is partially null controllable on $%
[0,t_{1}]$~ if and only if the family of exponent%
\begin{equation}
\left\{ e^{-\lambda _{j}t}b_{j},j\in I_{Y},t\in \left[ 0,t_{1}\right]
\right\}  \tag{-exp0}  \label{-exp0}
\end{equation}%
is strongly minimal.
\end{corollary}

\begin{corollary}
\label{C2}If the series (\ref{x_z(t1)} converges for any $x_{Z}^{0}\in Z,$
then system (\ref{CS1})--(\ref{P}) is partially null controllable on $%
[0,t_{1}]$~ if and only if equation%
\begin{equation*}
\dot{y}\left( t\right) =A_{Y}y\left( t\right) +b_{Y}u\left( t\right)
\end{equation*}%
is null (completely) controllable on $[0,t_{1}].$
\end{corollary}


\begin{proof}
Since the \ series (\ref{x_z(t1)}) converges for any $x_{Z}^{0}\in Z,$ and
the family $\psi _{j},~j\in Y$ forms a Riesz basic of $Y $, then for $%
x_{Y}^{0}\in Y~$and$~x_{Z}^{0}\in Z$ there exists $y^{0}\in Y, $ such that

\begin{equation*}
-\left( x^{0},\psi _{j}\right) +e^{-\lambda _{j}t_{1}}\left( S_{Z}\left(
t_{1}\right) x_{Z,}^{0},\psi _{Zj}\right) =-\left( y^{0},\psi _{j}\right)
,j\in Y.
\end{equation*}%
The corollary follows from the above formula and (\ref{LMP0}).
\end{proof}


\subsubsection{One sufficient condition for strong minimality}

The direct method for checking the strong minimality of any sequence (\ref{s}%
) is the computing of $\lim\limits_{n\rightarrow \infty }\lambda _{1}^{n}>0$
by numeric methods. It may be rather difficult. However if $b_{Y}\in Y,$
then one can find more simpler strong minimality conditions for sequence (%
\ref{-exp}) in some partial cases (see, for example, results of \cite%
{AmmarKhodja},\cite{Young80}, \cite{Young98},\cite{FattoriniRussel} and
references therein) .

Consider a strongly decreasing sequence of negative numbers $\{\lambda
_{j},j\in I_Y \}$, with finite upper density \cite{Mandelbrojt}.

Let the family 
\begin{equation}
\left\{ \exp (\lambda _{j}t),j\in I_{Y},t\in \left[ 0,t\right] \right\} 
\tag{exp(lambda\_t)}  \label{exp(lambda_t)}
\end{equation}%
be minimal, and let

\begin{equation}
\left\{ q_{j}\left( t\right) ,j\in I_{Y},t\in \left[ 0,t_{1}\right] \right\}
\tag{exp\_biort}  \label{exp_biort}
\end{equation}%
be biorthogonal family for (\ref{exp(lambda_t)}). 

\begin{theorem}
\label{T2} If $b_j \neq 0,\forall j\in I_Y, \lim_{j\rightarrow \infty }\frac{%
2\ln \left\vert b_{j}\right\vert }{\lambda _{j}}$ is finite, $%
t_{1}>\lim_{j\rightarrow \infty }\frac{2\ln \left\vert b_{j}\right\vert }{%
\lambda _{j}}$, and for any $\varepsilon >0$ there exists a positive
constant $K_{\varepsilon }$ such that 
\begin{equation}
\left\vert \left\vert q_{j}\left( \cdot \right) \right\vert \right\vert
<K_{\varepsilon }e^{-\lambda _{j}\varepsilon },j\in I_{Y},  \tag{bio\_est}
\label{bio_est}
\end{equation}
then the family 
\begin{equation}
\left\{ b_{j}\exp (-\lambda _{j}t),j\in I_{Y},t\in \left[ 0,t_{1}\right]
\right\}  \tag{-expb}  \label{-expb}
\end{equation}
is strongly minimal.
\end{theorem}


\begin{proof}
Obviously, the family 
\begin{equation*}
\left\{ \frac{q_{j}\left( t_{1}-t\right) }{b_{j}}e^{\lambda _{j}t_{1}},j\in
I_{Y},t\in \left[ 0,t_{1}\right] \right\}
\end{equation*}%
is biorthogonal to the family (\ref{-expb}). 

Let $I_{Y}^{f}$ be any finite subsequence of $I_{Y}.$ We have

\begin{equation*}
\left\Vert \sum_{j\in I_{Y}^{f}}\beta _{j}\frac{q_{j}\left( t_{1}-t\right)
e^{\lambda _{j}t_{1}}}{b_{j}}\right\Vert ^{2}\leq \left( \sum_{j\in
I_{Y}^{f}}\left\vert \beta _{j}\right\vert \frac{1}{\left\vert
b_{j}\right\vert }\left\Vert q_{j}\left( \cdot \right) \right\Vert
e^{\lambda _{j}t_{1}}\right) ^{2}\leq
\end{equation*}%
\begin{equation*}
\leq \left( \sum_{j\in I_{Y}^{f}}\left\vert \beta _{j}\right\vert
^{2}\right) \left( \sum_{j\in I_{Y}^{f}}\frac{1}{\left\vert b_{j}\right\vert
^{2}}\left\Vert q_{j}\left( \cdot \right) \right\Vert ^{2}e^{2\lambda
_{j}t_{1}}\right) \leq
\end{equation*}%
\begin{equation*}
\leq \left( \sum_{j\in I_{Y}^{f}}\left\vert \beta _{j}\right\vert
^{2}\right) \left( \sum_{j\in I_{Y}^{f}}\frac{1}{\left\vert b_{j}\right\vert
^{2}}K_{\varepsilon }^{2}e^{-2\lambda _{j}\varepsilon }e^{2\lambda
_{j}t_{1}}\right) =
\end{equation*}%
\begin{equation*}
=K_{\varepsilon }^{2}\left( \sum_{j\in I_{Y}^{f}}\left\vert \beta
_{j}\right\vert ^{2}\right) \left( \sum_{j\in I_{Y}^{f}}\frac{1}{\left\vert
b_{j}\right\vert ^{2}}e^{2\lambda _{j}\left( t_{1}-\varepsilon \right)
}\right) <K_{\varepsilon }^{2}\left( \sum_{j\in I_{Y}^{f}}\left\vert \beta
_{j}\right\vert ^{2}\right) \left( \sum_{j\in I_{Y}^{f}}\frac{1}{\left\vert
b_{j}\right\vert ^{2}}e^{\lambda _{j}t_{1}}\right)
\end{equation*}%
for $\varepsilon <\frac{t_{1}}{2}.$

The series $\sum_{j\in I_{Y}}\frac{1}{\left\vert b_{j}\right\vert ^{2}}%
e^{\lambda _{j}t_{1}}$ is the Dirichlet series with abscissa of \ the
convergence$~$\cite{Mandelbrojt}%
\begin{equation*}
\sigma _{c}=\lim_{j\rightarrow \infty }\frac{2\ln \left\vert
b_{j}\right\vert }{\lambda _{j}}.
\end{equation*}%
Hence

\begin{equation}
\left\Vert \sum_{j\in I_{Y}^{f}}\beta _{j}\frac{q_{j}\left( t_{1}-t\right) }{%
b_{j}}e^{\lambda _{j}t_{1}}\right\Vert ^{2}\leq M^{2}\sum_{j\in
I_{Y}^{f}}\left\vert \beta _{j}\right\vert ^{2},~\forall \beta _{j}\in 
\mathbb{R},j\in I_{Y}^{f}  \tag{BS}  \label{BS}
\end{equation}%
for$~$any$~t_{1}>\lim_{j\rightarrow \infty }\frac{2\ln \left\vert
b_{j}\right\vert }{\lambda _{j}},$where $M^{2}=K_{\varepsilon
}^{2}\sum_{j\in I_{Y}}\frac{1}{b_{j}^{2}}e^{\lambda _{j}t_{1}}.$ It follows
from (\ref{BS}), that the family $\left\{ \frac{q_{j}\left( t_{1}-t\right) }{%
b_{j}}e^{\lambda _{j}t_{1}},j\in I_{Y},t\in \left[ 0,t_{1}\right] \right\} $
is a Bessel sequence. Hence \cite{CasazzaChristensenLindner},~\cite{Young80}%
,~family (\ref{-expb}) is a Riesz-Fisher sequence, i.e. this family is
strongly minimal.

This proves the theorem.
\end{proof}

The next theorem follows from previous Theorem \ref{T2} and results of \cite%
{Mandelbrojt}, \cite{Leontiev}.

\begin{theorem}
\label{T3}. Let $I_{Y}=\{j_{1},j_{2},...,j_{n},...%
\},~j_{1}<j_{2}<...<j_{n}<..., ~\lambda _{j_{n}}<0,\left\vert
\lambda_{j_{n+1}}\right\vert >\left\vert \lambda _{j_{n}}\right\vert ,n\in 
\mathbb{N}.~$If:

\begin{enumerate}
\item there exists $\rho >0$ such that$~\left\vert \lambda
_{j_{1}}\right\vert >\rho ,$ $\left\vert \lambda _{j_{m}}-\lambda
_{j_{n}}\right\vert >\left\vert m-n\right\vert \rho ,$

\item $\sum_{n\in I_{Y}}\frac{1}{\left\vert \lambda _{n}\right\vert }%
<+\infty ,$

then the family (\ref{-exp}) is strongly minimal on $\left[ 0,t_{1}\right] $
\ for any%
\begin{equation*}
t_{1}>\lim\limits_{j\rightarrow \infty ,j\in I_{Y}}\frac{2\ln \left\vert
b_{j}\right\vert }{\lambda _{j}}.
\end{equation*}
\end{enumerate}
\end{theorem}

\begin{proof}
The proof is derived from the fact that conditions specified in the theorem
are adequate to establish the validity of (\ref{bio_est}).
\end{proof}

\begin{corollary}
\label{C3} \ If conditions of Theorem \ref{T3} hold and series (\ref{x_z(t1)}%
) converges for any $x_{Z}^{0}\in Z$, then system (\ref{CS1})--(\ref{P}) is
partially null controllable on $[0,t_{1}]$ for any%
\begin{equation*}
t_{1}>\lim\limits_{j\rightarrow \infty ,j\in I_{Y}}\frac{2\ln \left\vert
b_{j}\right\vert }{\lambda _{j}}.
\end{equation*}
\end{corollary}

\subsection{Example 1}

Consider the system of parabolic equations in the form (\ref{e1}):

\begin{equation}
\left\{ 
\begin{array}{c}
y_{t}^{\prime }=\triangle y+y+z+xu\left( t\right) , \\ 
z_{t}^{\prime }=z_{xx}^{\prime \prime },%
\end{array}%
\right. \left( t,x\right) \in \lbrack 0,+\infty )\times \left[ 0,\pi \right]
,  \tag{Ex1}  \label{Ex1}
\end{equation}

\begin{eqnarray}
y\left( 0,x\right) &=&y_{0}\left( x\right) ,z\left( 0,x\right) =z_{0}\left(
x\right) ,x\in L_{2}\left[ 0,\pi \right],  \notag \\
y\left( t,0\right) &=&y\left( t,\pi \right) =0,  \notag \\
z\left( t,0\right) &=&z\left( t,\pi \right) =0,  \notag
\end{eqnarray}%
where 
$y_0(x) ,z_0(x) \in L_{2}\left[ 0,\pi \right] ,$ $b\left( x\right) =\left( 
\begin{array}{c}
x \\ 
0%
\end{array}%
\right) ,u\left( t\right) \in L_{2}\left[ 0,t_{1}\right] .$

Let $P=\left( 
\begin{array}{cc}
1 & 0 \\ 
0 & 0%
\end{array}%
\right)$. In this case

$Y=\left\{ \left( 
\begin{array}{c}
y\left( x\right) \\ 
0%
\end{array}%
\right) ,y\left( \cdot \right) \in L_{2}\left[ 0,\pi \right] \right\},
Z=\left\{ \left( 
\begin{array}{c}
0 \\ 
z\left( x\right)%
\end{array}%
\right) ,z\left( \cdot \right) \in L_{2}\left[ 0,\pi \right] \right\},$

\begin{equation*}
A=\left( 
\begin{array}{cc}
\triangle +1 & 1 \\ 
0 & \vartriangle%
\end{array}%
\right) ,b=\left( 
\begin{array}{c}
x \\ 
0%
\end{array}%
\right) ,
\end{equation*}%
\begin{equation*}
A_{Y}=\vartriangle +1,A_{Z}=\vartriangle
\end{equation*}
with domains \cite[page 214, Example 4.6.6, page 241, Example 4.10.1]%
{Balakrishnan}

$D\left( A_{Y}\right) =D\left( A_{Z}\right) =\left\{ f,f^{\prime }~\mathrm{%
a.c.:}f^{\prime },f^{\prime \prime }\in L_{2}\left[ 0,\pi \right] \mathrm{%
~~and~}f\left( 0\right) =f\left( \pi \right) =0\right\} .$

After computations and enumeration we obtain:

1) $I_{Y}=\left\{ 1,3,5,...\right\} ,~I_{z}=\left\{ 2,4,6,...\right\} $

2) the eigenvalues $\lambda _{j}~$and eigenvectors $\varphi _{j},j\in 
\mathbb{N}~$of the operator $A$ are:

\begin{equation}
\lambda _{2m-1}=1-m^{2},\varphi _{2m-1}\left( x\right) =\sqrt{\frac{2}{\pi }}%
\left( 
\begin{array}{c}
\sin mx \\ 
0%
\end{array}%
\right) ,j=2m-1\in I_{Y},m\in \mathbb{N},  \tag{phiEY}  \label{phiEMY}
\end{equation}

\begin{equation}
\lambda _{2m}=-m^{2},\varphi _{2m}\left( x\right) =\sqrt{\frac{2}{\pi }}%
\left( 
\begin{array}{c}
-\sin mx \\ 
~~~\sin mx%
\end{array}%
\right) ,j=2m\in I_{Z},m\in \mathbb{N},  \tag{ phiEZ}  \label{phiEZ}
\end{equation}

3) the eigenvalues $\lambda _{i}$ and eigenvectors $\psi _{j}$ of the
operator $A^{\ast }$ are:

\begin{equation}
\lambda _{2m-1}=1-m^{2},\psi _{2m-1}\left( x\right) =\sqrt{\frac{2}{\pi }}%
\left( 
\begin{array}{c}
\sin mx \\ 
\sin mx%
\end{array}%
\right) ,j=2m-1\in I_{Y},~m\in \mathbb{N},  \tag{psiEY}  \label{psiEY}
\end{equation}
\begin{equation}
\lambda _{2m}=-m^{2},\psi _{2m}\left( x\right) =\sqrt{\frac{2}{\pi }}\left( 
\begin{array}{c}
0 \\ 
\sin mx%
\end{array}%
\right) ,j=2m\in I_{Z},~m\in \mathbb{N}.  \tag{psiEZ}  \label{psiEZ}
\end{equation}

Substituting (\ref{psiEY}) and (\ref{psiEZ}) to the formula $b_{j}=(b,\psi
_{j}),~j\in \mathbb{N}$, we obtain:

\begin{equation}
b_{j}=\left\{ 
\begin{tabular}{ll}
$\int_{0}^{\pi }x\sin mxdx=$ & $\frac{\left( -1\right) ^{m-1}}{m}\pi
,j=2m-1\in I_{Y},m\in \mathbb{N},$ \\ 
$~~0,~$ & $~~j=2m\in I_{Z},m\in \mathbb{N}.$%
\end{tabular}%
\right.  \tag{bj}  \label{bj}
\end{equation}

Using the characterization of Riesz bases via biorthogonal sequences
developed in \cite{Stoeva} (see Theorem 2.1), one can prove, that

\begin{lemma}
\label{L3}

The sequence

\begin{equation}
\left\{ \sqrt{\frac{2}{\pi }}\left( 
\begin{array}{c}
\sin mx \\ 
0%
\end{array}%
\right) ,\sqrt{\frac{2}{\pi }}\left( 
\begin{array}{c}
-\sin mx \\ 
~~~\sin mx%
\end{array}%
\right) ,~m\in \mathbb{N}\right\}  \tag{B1}  \label{B1}
\end{equation}%
and sequence 
\begin{equation}
\left\{ \sqrt{\frac{2}{\pi }}\left( 
\begin{array}{c}
\sin mx \\ 
~~~\sin mx%
\end{array}%
\right) ,\sqrt{\frac{2}{\pi }}\left( 
\begin{array}{c}
0 \\ 
~~~\sin mx%
\end{array}%
\right) ,~m\in \mathbb{N}\right\}  \tag{B1*}  \label{B1*}
\end{equation}%
are a Riesz basic of $L_{2}\left( \left[ 0,\pi \right] ,\mathbb{R}%
^{2}\right) .~$\newline
\end{lemma}

\begin{proof}
We will prove, that 1) sequence (\ref{B1}) is complete, and 2) sequences (%
\ref{B1}) and (\ref{B1*}) are Bessel sequences.

1) If for $f\left( x\right) =\left( 
\begin{array}{c}
f_{Y}\left( x\right) \\ 
f_{Z}\left( x\right)%
\end{array}%
\right) ,$ where $~f_{Y}\left( x\right) ,f_{Z}\left( x\right) \in L_{2}\left[
0,\pi \right] ,$ equalities $\left( \varphi _{2m-1},f\right) =0,~\left(
\varphi _{2m},f\right) =0$ hold for any $m\in \mathbb{N}$, i.e. 
\begin{eqnarray*}
\left( \sin mx,f_{Y}\right) &=&0, \\
-\left( \sin mx,f_{Y}\right) +\left( \sin mx,f_{Z}\right) &=&0, \\
\forall m &\in &\mathbb{N},
\end{eqnarray*}%
then from the completeness of the sequence $\{\sin mx, m\in \mathbb{N}\}$ in 
$L_{2}\left[ 0,\pi \right] ,$ it follows, $f_{Y}=f_{Z}=0.$

Hence, sequence (\ref{B1}) is complete.

2) For any $f\left( x\right) =\left( 
\begin{array}{c}
f_{Y}\left( x\right) \\ 
f_{Z}\left( x\right)%
\end{array}%
\right) ,$ where$~f_{Y}\left( x\right) ,f_{Z}\left( x\right) \in L_{2}\left[
0,\pi \right] ,$ we have

\begin{equation*}
\sum_{m=1}^{\infty }\left\vert \left( \varphi _{2m-1},f\right) \right\vert
^{2}+\sum_{m=1}^{\infty }\left\vert \left( \varphi _{2m},f\right)
\right\vert ^{2}\leq
\end{equation*}

\begin{equation*}
\leq 3\sum_{m=1}^{\infty }\left\vert \left( \sqrt{\frac{2}{\pi }}\sin
mx,f_{Y}\right) \right\vert ^{2}+2\sum_{m=1}^{\infty }\left\vert \left( 
\sqrt{\frac{2}{\pi }}\sin mx,f_{Z},\right) \right\vert ^{2}<+\infty .
\end{equation*}%
\newline

Hence, sequence (\ref{B1}) is a Bessel sequence.

By the same way, for any $f\left( x\right) =\left( 
\begin{array}{c}
f_{Y}\left( x\right) \\ 
f_{Z}\left( x\right)%
\end{array}%
\right) ,$ where$~f_{Y}\left( x\right) ,f_{Z}\left( x\right) \in L_{2}\left[
0,\pi \right] ,$ we have%
\begin{equation*}
\sum_{m=1}^{\infty }\left\vert \left( \psi _{2m-1},f\right) \right\vert
^{2}+\sum_{m=1}^{\infty }\left\vert \left( \psi _{2m},f\right) \right\vert
^{2}\leq
\end{equation*}

\begin{equation*}
\leq 2\sum_{m=1}^{\infty }\left\vert \left( \sqrt{\frac{2}{\pi }}\sin
mx,f_{Y}\right) \right\vert ^{2}+3\sum_{m=1}^{\infty }\left\vert \left( 
\sqrt{\frac{2}{\pi }}\sin mx,f_{Z}\right) \right\vert ^{2}<+\infty .
\end{equation*}%
Hence the sequence (\ref{B1*}) is a Bessel sequence also.

This proves the lemma.
\end{proof}

Lemma \ref{L3} shows that equation (\ref{Ex1}) satisfies all the conditions
of equation (\ref{CS1}).

\medskip For $j=2m-1,m\in \mathbb{N},$ we have $e^{-2\lambda
_{j}t_{1}}\left( S_{Z}\left( t_{1}\right) x_{Z}^{0},\psi _{Zj}\right) =$

$=e^{-\lambda _{2m-1}t_{1}}\left( \sum_{n=1}^{\infty }\left( x_{Z\left(
2n\right) },\psi _{Z\left( 2n\right) }\right) e^{\lambda _{2n}t_{1}}\varphi
_{Z\left( 2n\right) },\psi _{Z\left( 2m-1\right) }\right) =$

$=e^{\left( m^{2}-1\right) t_{1}}\left( \dint_{0}^{\pi }\left(
\sum_{n=1}^{\infty }\sqrt{\frac{2}{\pi }}\left( \dint_{0}^{\pi }z^{0}\left(
\xi \right) \sin n\xi d\xi \right) e^{-n^{2}t_{1}}\sqrt{\frac{2}{\pi }}\sin
nx\right) \sqrt{\frac{2}{\pi }}\sin mx\right) dx=$

$=e^{\left( m^{2}-1\right) t_{1}}e^{-m^{2}t_{1}}\frac{2}{\pi }\left(
\dint_{0}^{\pi }z^{0}\left( x\right) \sin mxdx\right) dx=\sqrt{\frac{2}{\pi }%
}e^{-t_{1}}\left( \dint_{0}^{\pi }z^{0}\left( x\right) \sin mdx\right) ,$

so

\begin{equation*}
\sum_{\substack{ j=1,  \\ j=1,3,5,...}}^{\infty }\left\vert e^{-\lambda
_{j}t_{1}}\left( S_{Z}\left( t_{1}\right) x_{Z}^{0},\psi _{Zj}\right)
\right\vert ^{2}=
\end{equation*}

\begin{equation*}
=\frac{2}{\pi }e^{-2t_{1}}\sum_{m=1}^{\infty }\left\vert \left(
\int_{0}^{\pi }z^{0}\left( x\right) \sin mxdx\right) dx\right\vert
^{2}<+\infty .
\end{equation*}

Also we have

$e^{-\lambda _{j}t_{1}}b_{j}=e^{\left( m^{2}-1\right) t_{1}}\frac{\left(
-1\right) ^{m-1}}{m}\pi =\pi e^{-1}\frac{\left( -1\right) ^{m-1}}{m}%
e^{m^{2}t_{1}},j=2m-1\in I_{Y},m\in \mathbb{N}.$

When the family

\begin{equation*}
\left\{ \frac{\left( -1\right) ^{m-1}}{m}e^{m^{2}t_{1}},m=1,2,...,.\right\}
\end{equation*}%
is strongly minimal?

We have $\lim\limits_{m\rightarrow \infty }\ln \left( \frac{1}{m}\right) 
\frac{1}{-m^{2}}=\allowbreak \lim\limits_{m\rightarrow \infty }\frac{1}{%
\left( \ln m\right) m^{2}}$ $=0.$

In accordance with Theorem 3 system ~(\ref{Ex1}) is partially
null-controllable on $[0,t_{1}]$~for any $t_{1}>0.\ $

However there is no $t_{1}>0,$ such that system (\ref{Ex1})\ is completely
null-controllable on $[0,t_{1}]$~, because the second equation of this
system doesn't depend on a control.

\bigskip

\section{The case $\protect\sigma _{YZ}\neq \varnothing $}

\subsection{Eigenvectors of the operators $A$ and $A^{\ast }$}

If $j\in I_{Y}\backslash I_{YZ}$ and $k\in I_{Z}\backslash I_{YZ},~$we are
in conditions of the previous section, so eigenvectors of the operator $A$
and $A^{\ast }$ are obtained in the previous section andthey satisfy satisfy
orthogonality conditions (\ref{ortYZ}).

If $m\in $ $I_{YZ},$ the eigenvalue $\lambda _{m}$ of the operator 
\begin{equation*}
A=\left( 
\begin{array}{cc}
A_{Y} & A_{YZ} \\ 
0 & A_{Z}%
\end{array}%
\right)
\end{equation*}

may have two linear independent eigenvectors.

The first one is $\phi _{1m}$ $=\left( 
\begin{array}{c}
\phi _{1Ym} \\ 
0%
\end{array}%
\right) ,~$ where $\phi _{1Ym}\neq 0\ $is an eigenvector of the operator $%
A_{Y}.$

The second one is: \textrm{$\phi $}$_{2m}=\left( 
\begin{array}{c}
\mathrm{\phi }_{2Ym} \\ 
\phi _{2Zm}%
\end{array}%
\right) ,~$where $\phi _{2Zm}$ is an eigenvector of the operator $A_{Z},~$
provided that there exists a solution of the operator equation 
\begin{equation}
A_{Y}\mathrm{\phi }_{2Ym}+A_{YZ}\phi _{2Zm}=\lambda _{m}\mathrm{\phi }_{2Ym}
\tag{AYZ}  \label{AYZ}
\end{equation}%
for $m\in $ $I_{YZ}.$ In this case the vector $\mathrm{\phi }_{2Ym}$\ \ is a
solution of above equation.

Since $\lambda _{m}\ $is an eigenvalue of the operator $A_{Y},$ the operator
equation (\ref{AYZ}) may be not solvable. In this case the operator $A$ has
an unique eigenvector $\phi _{1m}$ $=\left( 
\begin{array}{c}
\phi _{1Ym} \\ 
0%
\end{array}%
\right) ~$

The same considerations can be performed for the operator%
\begin{equation*}
A^{\ast }=\left( 
\begin{array}{cc}
A_{Y}^{\ast }, & 0 \\ 
A_{YZ}^{\ast } & A_{Z}^{\ast }%
\end{array}%
\right) ,
\end{equation*}
i.e .for $m\in $ $I_{YZ}$ the eigenvalue $\bar{\lambda}_{m}$ may have two
linear independent eigenvectors.

The first one is \textrm{$\psi $}$_{2m}\mathrm{=}\left( 
\begin{array}{c}
\mathrm{0} \\ 
\psi _{2Zm}%
\end{array}%
\right) ,~$ where $\phi _{1Ym}\neq 0\ $is an eigenvector of the operator $%
A_{Z}^{\ast }.$

The second one is: \textrm{$\psi $}$_{1m}\mathrm{=}\left( 
\begin{array}{c}
\psi _{1Ym} \\ 
\mathrm{\psi }_{1Zm}%
\end{array}%
\right) ,~$where $\psi _{1Ym}$ is an eigenvector of the operator $%
A_{Y}^{\ast },~$ provided that there exists a solution the operator equation 
\begin{equation}
A_{YZ}^{\ast }\psi _{1Ym}+A_{Z}^{\ast }\mathrm{\psi }_{1Zm}=\bar{\lambda}_{m}%
\mathrm{\psi }_{1Zm}.  \tag{AYZ*}  \label{AYZ*}
\end{equation}

Since $\bar{\lambda}_{m}\ $is an eigenvalue of the operator $A_{Z}^{\ast },$
the operator equation (\ref{AYZ*}) may be not solvable. In this case the
operator $A$ has an unique eigenvector $\psi _{2m}=\left( 
\begin{array}{c}
0 \\ 
\psi _{2Zm}%
\end{array}%
\right) .$

%

If equation (\ref{AYZ}) does not admit a solution, then the vector $\phi
_{1m}$ in n unique eigenvector of the \ operator $A$ for the eigenvalue $%
\lambda _{m.}.$

In this case consider the equation

\begin{equation}
\left( A_{Y}-\lambda _{m}I\right) \mathrm{\phi }_{2Ym}=c_{m}\phi
_{1Ym}-A_{YZ}\phi _{2Zm}.  \tag{AYZG}  \label{AYZG}
\end{equation}

\begin{theorem}
\bigskip \label{T4}
\end{theorem}

The operator equation

\begin{equation}
A\mathrm{\phi }_{2m}-\lambda _{m}\mathrm{\phi }_{2m}=c_{m}\phi _{1m}, 
\tag{gev}  \label{gev}
\end{equation}%
has a solution \textrm{$\phi $}$_{2m}=\left( 
\begin{array}{c}
\mathrm{\phi }_{2Ym} \\ 
\phi _{2Zm}%
\end{array}%
\right) $ if and only if

\begin{equation}
c_{m}=\left( A_{YZ}\phi _{2Zm},\psi _{1Ym}\right) .  \tag{c}  \label{c}
\end{equation}

%
%
%
%

\begin{proof}
Equation (\ref{gev}) is equivalent to equation (\ref{AYZG}). 

Equation (\ref{AYZG}) has a solution \ if and only if \cite{Krein1971}

\begin{equation*}
\left( c\phi _{1Ym}-A_{YZ}\phi _{2Zm},\psi _{1Ym}\right) =0.
\end{equation*}

We have 
\begin{equation*}
\left( c_{m}\phi _{1Ym}-A_{YZ}\phi _{2Zm},\psi _{1Ym}\right) =
\end{equation*}%
\begin{eqnarray*}
&=&c_{m}\left( \phi _{1Ym},\psi _{1Ym}\right) -\left( A_{YZ}\phi _{2Zm},\psi
_{1Ym}\right) = \\
&=&c_{m}-\left( A_{YZ}\phi _{2Zm},\psi _{1Ym}\right) ,
\end{eqnarray*}%
i.e. the last condition is equivalent to the condition (\ref{c}).

This proves the theorem.
\end{proof}

If $c_{m}=0,$ then \textrm{$\phi $}$_{2m}\ $is a solution of the equation%
\begin{equation*}
A\mathrm{\phi }_{2m}-\lambda _{m}\mathrm{\phi }_{2m}=0,
\end{equation*}%
so \textrm{$\phi $}$_{2m}$ is the second eigenvector of the operator $A$ of
the same simple eigenvalue $\lambda _{m}$ having the multiplicity $2.$

If $c_{m}\neq 0$, then the \textrm{$\phi $}$_{2m}$ is a generalized
eigenvector of the operator $A$, corresponding to the eigenvalue $\lambda
_{m}$, having the multiplicity $2$. \bigskip

If equation (\ref{AYZ*}) does not admit a solution, then the vector \textrm{$%
\psi $}$_{2m}$ in an unique eigenvector of the \ operator $A$ for the
eigenvalue $\lambda _{m.}.$

In this case consider the equation

\begin{equation}
\left( A_{Z}^{\ast }-\bar{\lambda}_{m}I\right) \mathrm{\psi }%
_{1Zm}=c_{m}\psi _{2Zm}-A_{YZ}^{\ast }\psi _{1Ym},  \tag{AYZG*}
\label{AYZG*}
\end{equation}

\begin{theorem}
The operator equation 
\begin{equation}
A^{\ast }\mathrm{\psi }_{1m}-\lambda _{m}\mathrm{\psi }_{1m}=c_{m}\psi _{2m},
\tag{gev*}  \label{gev*}
\end{equation}%
has a solution \textrm{$\psi $}$_{1m}=\left( 
\begin{array}{c}
\psi _{1Ym} \\ 
\mathrm{\psi }_{1Zm}%
\end{array}%
\right) ,~$where $\mathrm{\phi }_{2Zm}\neq 0$ is an eigenvector of the
operator$A_{Z},\ $corresponding to its eigenvalue $\lambda _{m},$ if and
only if 
\begin{equation}
c_{m}=\left( \phi _{2Zm},A_{YZ}^{\ast }\psi _{1Ym}\right) .  \tag{c*}
\label{c*}
\end{equation}
\end{theorem}


\begin{proof}
The operator equation (\ref{gev*}) is equivalent to operator equation (\ref%
{AYZG*}). Equation (\ref{AYZG*}) has a solution if and only if \cite%
{Krein1971}%
\begin{equation*}
\left( \phi _{2Zm},c_{m}\psi _{2Zm}-A_{YZ}^{\ast }\psi _{1Ym}\right) =0.
\end{equation*}

We have $\left( \phi _{2Zm},c_{m}\psi _{2Zm}-A_{YZ}^{\ast }\psi
_{1Ym}\right) =c_{m}\left( \phi _{2Zm},\psi _{2Zm}\right) -\left( \phi
_{2Zm},A_{YZ}^{\ast }\psi _{1Ym}\right) .$ i.e. the last condition is
equivalent to the condition (\ref{c*}).

This proves the theorem.
\end{proof}

Conditions (\ref{c}) and (\ref{c*}) are equivalent.

If $c_{m}=0,$ then we obtain equation

\begin{equation}
A^{\ast }\mathrm{\psi }_{1m}-\lambda _{m}\mathrm{\psi }_{1m}=0  \tag{ev*}
\end{equation}%
so the vector $~$\textrm{$\psi $}$_{1m}\mathrm{=}\left( 
\begin{array}{c}
\psi _{1Ym} \\ 
\mathrm{\psi }_{1Zm}%
\end{array}%
\right) $ is a second eigenvector of the operator $A^{\ast },$ corresponding
the same simple eigenvalue $\bar{\lambda}_{m.}$of \ the multiplicity $2$.

If $c_{m}\neq 0,$ then $~$\textrm{$\psi $}$_{1m}$ is a generalized
eigenvector of the operator $A^{\ast },$ corresponding the eigenvalue $\bar{%
\lambda}_{m.}$ of \ the multiplicity $2$.

\begin{lemma}
\label{L2}The sequence $\left\{ c_{m},m\in I_{YZ}\right\} $ is bounded.
\end{lemma}

\begin{proof}
We have 
\begin{equation*}
\left\vert c_{m}\right\vert =\left\vert \left( A_{YZ}\phi _{2Zm},\psi
_{1Ym}\right) \right\vert \leq \left\Vert A_{YZ}\right\Vert \left\Vert \phi
_{2Zm}\right\Vert \left\Vert \psi _{1Ym}\right\Vert ,m\in I_{YZ}.
\end{equation*}

The sequences $\left\{ \phi _{2Zm},m\in I_{YZ}\right\} $ and $\left\{ \psi
_{1Ym},m\in I_{YZ}\right\} $ are bounded, because in accordance with Lemma %
\ref{L1} the first one is a Riesz basic of $Z,$ and the second one is a
Riesz basic of $Y.$

This proves the Lemma.
\end{proof}

The eigenvectors $\phi _{1m},~\psi _{2m}~$ and the generalized eigenvectors 
\textrm{$\phi $}$_{2m},$ \textrm{$\psi $}$_{1m}$ of the operators $A$ and $%
A^{\ast }$correspondingly are chosen such that the following orthogonality
conditions hold:


\begin{equation}
\left( 
\begin{array}{c}
\begin{array}{cc}
\left( \phi _{1l},\mathrm{\psi }_{1j}\right) =\delta _{lj} & \left( \phi
_{1l},\psi _{2j}\right) =0,%
\end{array}
\\ 
\begin{array}{cc}
\left( \mathrm{\phi }_{2l},\mathrm{\psi }_{1j}\right) =0 & \left( \mathrm{%
\phi }_{2l},\psi _{2j}\right) =\delta _{lj},%
\end{array}%
\end{array}%
\right) ,l,j\in I_{YZ}.  \tag{Y}  \label{Y}
\end{equation}%
Hence, $\left( \phi _{2l},\mathrm{\psi }_{1j}\right) =\left( \mathrm{\phi }%
_{2Yl},\psi _{1Yj}\right) +\left( \mathbf{\phi }_{2Zl},\mathrm{\psi }%
_{1Zj}\right) =0,~l,j\in I_{YZ}\iff $ 
\begin{equation}
\left( \mathrm{\phi }_{2Yl},\mathbf{\psi }_{1Yj}\right) =-\left( \phi _{2Zl},%
\mathrm{\psi }_{1Zj}\right) ,~l,j\in I_{YZ}.  \tag{YZ}  \label{YZ}
\end{equation}

Lemma \ref{L1} holds true for the case $\sigma _{YZ}\neq \varnothing $ as
well.

\section{ Partial controllability conditions}


Denote: 

$b_{1j}=\left\{ 
\begin{tabular}{ll}
$\left( b,\psi _{j}\right) ,$ & $j\in I_{Y}\backslash I_{YZ},$ \\ 
$\left( b,\mathrm{\psi }_{1j}\right) ,$ & $j\in I_{YZ},$%
\end{tabular}%
\right. ,b_{2j}=\left( b,\mathrm{\psi }_{2j}\right) ,j\in I_{YZ}.$

\bigskip

$b_{j}\left( \tau \right) =\left\{ 
\begin{array}{c}
\begin{tabular}{ll}
$b_{1j}-e^{-\lambda _{j}\left( t_{1}-\tau \right) }\left( S_{Z}\left(
t_{1}-\tau \right) b_{Z},\psi _{Zj}\right) ,$ & $j\in I_{Y}\backslash
I_{YZ}, $ \\ 
$b_{1j}+b_{2j}c_{j}\left( t_{1}-\tau \right) -$ & 
\end{tabular}
\\ 
-e^{-\lambda _{j}\left( t_{1}-\tau \right) }\left( S_{Z}\left( t_{1}-\tau
\right) b_{Z},\mathrm{\psi }_{1Zj}\right) .j\in I_{YZ}.%
\end{array}%
\right. $

\begin{theorem}
\label{T5} Let the series 
\begin{equation}
\sum\limits_{j\in I_{Y}}e^{-2\lambda _{j}t_{1}}\left\vert \left( S_{Z}\left(
t_{1}\right) x_{Z}^{0},\psi _{Zj}\right) \right\vert ^{2}  \tag{xg\_z(t1)}
\label{xg_z(t1)}
\end{equation}%
converges for any $x_{Z}^{0}\in Z.$

Equation (\ref{CS1})--(\ref{P}) is partially null controllable on $[0,t_{1}]$%
~ if and only if the family of exponent%
\begin{equation}
\left\{ e^{-\lambda _{j}t}b_{j}\left( t\right) ,j\in I_{Y},t\in \left[
0,t_{1}\right] \right\}  \tag{-expg}  \label{-expg}
\end{equation}%
is strongly minimal.
\end{theorem}


\begin{proof}
Since the family of eigenvectors $\left\{ \varphi _{i},i\in \mathbb{N}%
\right\} ~$of the operator $A$ is a Riesz basic of $X$ and the family of
eigenvectors $\left\{\psi _{i},i\in \mathbb{N}\right\} ~$of the operator $%
A^{\ast }$ is a Riesz basic of $X$ as well , for any $x^{0}=\left( 
\begin{array}{c}
x_{Y}^{0} \\ 
x_{z}^{0}%
\end{array}%
\right) \in X=Y\times Z$ we have

$x^0(t_1):=S\left( t_{1}\right) x^{0}=\sum\limits_{l\in I_{Y}\backslash
I_{YZ}}e^{\lambda _{l}t_{1}}\left( x^{0},\psi _{l}\right) \left( 
\begin{array}{c}
\varphi _{Yl} \\ 
0%
\end{array}%
\right) +$

$+\sum\limits_{l\in I_{YZ}}e^{\lambda _{l}t_{1}}\left[\left( x^{0},\psi
_{1l}\right) \left( 
\begin{array}{c}
\phi _{1Yl} \\ 
0%
\end{array}%
\right) +\left( x^{0},\psi _{2l}\right) \left( c_{j}t_{1}\left( 
\begin{array}{c}
\phi _{1Yl} \\ 
0%
\end{array}%
\right) +\left( 
\begin{array}{c}
\phi _{2Yl} \\ 
\phi _{2Zl}%
\end{array}%
\right) \right)\right] +$


$+\sum\limits_{k\in I_{Z}|I_{YZ}}\left( x^{0},\psi _{k}\right) e^{\lambda
_{k}t_{1}}\left( 
\begin{array}{c}
\varphi _{Yk} \\ 
\varphi _{Zk}%
\end{array}%
\right) ,$ \bigskip \bigskip

$x_{Y}^{0}\left( t_{1}\right):=PS(t_1)x^0 =\sum\limits_{l\in I_{Y}\backslash
I_{YZ}}e^{\lambda _{l}t_{1}}\left( x^{0},\psi _{l}\right) \varphi _{Yl}+$

$+\sum\limits_{l\in I_{YZ,}}e^{\lambda _{l}t_{1}}\left[ \left( x^{0},\psi
_{1l}\right) \phi _{1Yl}+\left( x^{0},\psi _{2l}\right) \left(
c_{j}t_{1}\phi _{1Yl}+\phi _{2Yl}\right) \right] +$

$+\sum\limits_{k\in I_{Z}|I_{YZ}}e^{\lambda _{k}t_{1}}\left( x^{0},\psi
_{k}\right) \varphi _{Yk}.$ \bigskip \bigskip

Using (\ref{ortY})--(\ref{ortYZ}), (\ref{Y})--(\ref{YZ}), we obtain the
following:

If $j\in I_{Y}\backslash I_{YZ},$ then

$\left( x_{Y}^{0}\left( t_{1}\right) ,\psi _{Yj}\right) =\sum\limits_{l\in
I_{Y}\backslash I_{YZ}}\left( x^{0},\psi _{j}\right) e^{\lambda
_{lk}t_{1}}\left( \varphi _{Yl},\psi _{Yj}\right) +$

$+\sum\limits_{l\in I_{YZ}}e^{\lambda _{l}t_{1}}\left( x^{0},\psi
_{2l}\right) \left[ c_{m}t_{1}\left( \phi _{1Yl},\psi _{Yj}\right) +\left(
\phi _{2Yl},\psi _{Yj}\right) \right] +$

$+\sum\limits_{l\in I_{YZ},}e^{\lambda _{l}t_{1}}\left( x^{0},\psi
_{1l}\right) \left( \phi _{1Yl},\psi _{Yj}\right) +$

$+\sum\limits_{k\in I_{Z}|I_{YZ}}\left( x^{0},\psi _{k}\right) e^{\lambda
_{k}t_{1}}\left( \varphi _{Yk},\psi _{Yj}\right) =$ \medskip

$=e^{\lambda _{j}t_{1}}\left( x^{0},\psi _{Yj}\right) +$

$+\sum\limits_{l\in I_{YZ}}\left( x_{Z}^{0},\psi _{2Zl}\right) e^{\lambda
_{l}t_{1}}\left( \phi _{2Yl},\psi _{Yj}\right) +$

$+\sum\limits_{k\in I_{Z}|I_{YZ}}\left( x^{0},\psi _{Zk}\right) e^{\lambda
_{k}t_{1}}\left( \varphi _{Yk},\psi _{Yj}\right) =$ \medskip

$=e^{\lambda _{j}t_{1}}\left( x^{0},\psi _{j}\right) -$

$-\sum\limits_{I_{Z}|I_{YZ}}\left( \ x^{0},\psi _{k}\right) \ e^{\lambda
_{k}t_{1}}(\varphi _{Zk},\psi _{Zj})-\sum\limits_{k\in I_{Z}}\left( \
x^{0},\psi _{k}\right) e^{\lambda _{k}t_{1}}\left( \varphi _{2Zk},\psi
_{Zj}\right) ,$ and since the sequence $\left\{ \varphi _{2Zk,}k\in
I_{Z}\right\} $ is a Riesz basic of $Z,$ we obtain, that

$\sum\limits_{k\in I_{Z}}\left( \ x_{Z}^{0},\psi _{Zk}\right) \ e^{\lambda
_{k}t_{1}}\varphi _{Zk}=S_{Z}\left( t_{1}\right) x_{Z}^{0},$ so%
\begin{equation*}
\left( x_{Y}\left( t_{1}\right) ,\psi _{Yj}\right) =e^{\lambda
_{j}t_{1}}\left( x^{0},\psi _{j}\right) -\left( S_{Z}\left( t_{1}\right)
x_{Z}^{0},\psi _{Zj}\right) .
\end{equation*}

\bigskip If $j\in I_{YZ},$ then

$\left( PS\left( t_{1}\right) x^{0},\mathrm{\psi }_{1Yj}\right) =$%
%
%
%

$=\sum\limits_{l\in I_{YZ}}e^{\lambda _{l}t_{1}}\left( x^{0},\mathrm{\psi }%
_{1l}\right) \left( \mathrm{\phi }_{1Yl},\mathrm{\psi }_{1Yj}\right) +$

$+\sum\limits_{l\in I_{YZ}}e^{\lambda _{l}t_{1}}\left( x^{0},\mathrm{\psi }%
_{2l}\right) \left( c_{j}t_{1}\left( \mathrm{\phi }_{1Yl},\mathrm{\psi }%
_{1Yj}\right) +\left( \mathrm{\phi }_{2Yl},\mathrm{\psi }_{1Yj}\right)
\right) +$

$+\sum\limits_{k\in I_{Z}|I_{YZ}}e^{\lambda _{k}t_{1}}\left( x^{0},\psi
_{k}\right) \left( \varphi _{Yk},\mathrm{\psi }_{1Yj}\right) =$ \medskip

$=e^{\lambda _{j}t_{1}}\left[ \left( x^{0},\mathrm{\psi }_{1j}\right)
+c_{j}t_{1}\left( x^{0},\mathrm{\psi }_{2l}\right) \right]
+\sum\limits_{l\in I_{YZ}}e^{\lambda _{l}t_{1}}\left( x^{0},\mathrm{\psi }%
_{2l}\right) \left( \mathrm{\phi }_{2Yl},\mathrm{\psi }_{1Yj}\right) $

$+\sum\limits_{k\in I_{Z}|I_{YZ}}e^{\lambda _{k}t_{1}}\left( x^{0},\psi
_{k}\right) \left( \varphi _{Yk},\mathrm{\psi }_{1Yj}\right) = e^{\lambda
_{j}t_{1}}\left[ \left( x^{0},\mathrm{\psi }_{1j}\right) +c_{j}t_{1}\left(
x^{0},\mathrm{\psi }_{2l}\right) \right] -$

$-\sum\limits_{l\in I_{YZ}}e^{\lambda _{l}t_{1}}\left( x^{0},\mathrm{\psi }%
_{2l}\right) \left( \mathrm{\phi }_{2Zl},\mathrm{\psi }_{1Zj}\right)
-\sum\limits_{k\in I_{Z}|I_{YZ}}e^{\lambda _{k}t_{1}}\left( x^{0},\psi
_{k}\right) \left( \varphi _{Zk},\mathrm{\psi }_{1Zj}\right) ,$ and again,
since the sequence $\left\{ \varphi _{2Zk,}k\in I_{Z}\right\} $ is a Riesz
basic of $Z,$ we obtain, that%
\begin{equation*}
\left( x_{Y}\left( t_{1}\right) ,\psi _{Yj}\right) =e^{\lambda jt_{1}}\left[
\left( x^{0},\mathrm{\psi }_{1j}\right) +\left( x^{0},\mathrm{\psi }%
_{2j}\right) c_{j}t_{1}\right] -\left( S_{Z}\left( t_{1}\right) x_{Z}^{0},%
\mathrm{\psi }_{1Zj}\right) .
\end{equation*}

\bigskip Using $b, t_{1}-\tau $ instead of $x^{0}, t_{1}$ in above
computations, we obtain :

%
%
%
%
%

if $j\in I_{Y}\backslash I_{YZ},$ then

$\left( PS\left( t_{1}-\tau \right) b,\psi _{j}\right) =e^{\lambda
_{j}\left( t_{1}-\tau \right) }b_{1j}-\left( S_{Z}\left( t_{1}\right)
b_{Z},\psi _{Zj}\right) ,$

if$~j\in I_{YZ},$ then

$\left( PS\left( t_{1}-\tau \right) b,\mathrm{\psi }_{1j}\right) =$

$=\left[ b_{1j}+b_{2j}c_{j}\left( t_{1}-\tau \right) \right] e^{\lambda
_{j}\left( t_{1}-\tau \right) }-\left( S_{Z}\left( t_{1}-\tau \right) b_{Z},%
\mathrm{\psi }_{1Zj}\right) .$

\bigskip

After integration in $\left[ 0,t_{1}\right] $ we obtain:

\medskip if $j\in I_{Y}\backslash I_{YZ},$ then

$\left(x_Y(t_1),\psi _{j}\right)=e^{\lambda _{j}t_{1}}\left( x^{0},\psi
_{j}\right) -\left( S_{Z}\left(t_{1}\right) x_{Z}^{0},\psi _{Zj}\right) +$

$+\int_{0}^{t_{1}}\left[ b_{1j}e^{\lambda _{j}\left( t_{1}-\tau \right)
}-\left( S_{Z}\left( t_{1}-\tau \right) b_{Z},\psi _{Zj}\right) \right]
u\left( \tau \right) d\tau ;$

\medskip if $j\in I_{YZ},$ then

$\left(x_Y(t_1),\psi _{j}\right)=\left[ \left( x^{0},\mathrm{\psi }%
_{1j}\right) +\left( x^{0},\mathrm{\psi }_{2j}\right) c_{j}t_{1}\right]
e^{\lambda _{j}t_{1}}-\left( S_{Z}\left( t_{1}\right) x_{Z}^{0},\mathrm{\psi 
}_{2Zj}\right) +$

$+\int_{0}^{t_{1}}\left( \left[ b_{1j}+b_{2j}c_{j}\left( t_{1}-\tau \right) %
\right] e^{\lambda _{j}\left( t_{1}-\tau \right) }-\left( S_{Z}\left(
t_{1}-\tau \right) b_{Z},\mathrm{\psi }_{2Zj}\right) \right) u\left( \tau
\right) d\tau .$

Hence, equation (\ref{CS1})--(\ref{P}) is partial controllable on $[0,t_1]$
if and only if the moment problem 
\begin{equation}
-\left( x^{0},\psi _{j}\right) +e^{-\lambda _{j}t_{1}}\left( S_{Z}\left(
t_{1}\right) x_{Z}^{0},\psi _{Zj}\right) =  \tag{LPRY}  \label{LPRY}
\end{equation}%
\begin{equation*}
=\int_{t_{0}}^{t_{1}}\left[ b_{1j}e^{-\lambda _{j}\tau }-e^{-\lambda
_{j}t_{1}}\left( S_{Z}\left( t_{1}-\tau \right) b_{Z},\psi _{Zj}\right) %
\right] u\left( \tau \right) d\tau , \\
j\in I_{Y}\backslash I_{YZ},  \notag
\end{equation*}

\begin{equation}
-\left( x^{0},\mathrm{\psi }_{1j}\right) -\left( x^{0},\mathrm{\psi }%
_{2j}\right) c_{j}t_{1}+e^{-\lambda _{j}t_{1}}\left( S_{Z}\left(
t_{1}\right) x_{Z}^{0},\mathrm{\psi }_{2Zj}\right) =  \notag
\end{equation}

\begin{equation}
=\int_{0}^{t_{1}}\left[ b_{1j}+b_{2j}c_{j}\left( t_{1}-\tau \right) \right]
e^{-\lambda _{j}\tau }u\left( \tau \right) d\tau -  \tag{LPRYZ}
\label{LPRYZ}
\end{equation}%
\begin{equation}
-\int_{0}^{t_{1}}e^{-\lambda _{j}t_{1}}\left( S_{Z}\left( t_{1}-\tau \right)
b_{Z},\mathrm{\psi }_{1Zj}\right) u\left( \tau \right) d\tau ,  \notag \\
j\in I_{YZ}.  \notag
\end{equation}
has a solution $u(\cdot)\in L_2[0,t_1]$.

Since the series (\ref{xg_z(t1)}) converges for any $x_{z}^{0}\in Z,$ the
sequence

$\left\{ \left( x^{0},\psi _{j}\right) -e^{-\lambda _{j}t_{1}}\left(
S_{Z}\left( t_{1}\right) x_{Z}^{0},\psi _{Zj}\right) ,j\in I_{Y}\backslash
I_{YZ}\right\} \in l^{2},$

\noindent and, in accordance with Lemma \ref{L2}, the sequence $\left\{
c_{j,}j\in I_{YZ}\right\} $ is bounded, so

$\left\{ \left( x^{0},\mathrm{\psi }_{1j}\right) +\left( x^{0},\mathrm{\psi }%
_{2j}\right) c_{j}t_{1}-e^{-\lambda _{j}t_{1}}\left( S_{Z}\left(
t_{1}\right) x_{Z}^{0},\mathrm{\psi }_{2Zm}\right) ,j\in I_{YZ}\right\} \in
l^{2}$ as well.

\noindent Hence \cite{Boaz}, \cite{Young98} above problem moment has a
solution if and only if family of exponential (\ref{-expg}) is strongly
minimal.

This proves the theorem.
\end{proof}

If $b\in Y,~$i.e. $b=\left( 
\begin{array}{c}
b_{Y} \\ 
0%
\end{array}%
\right) ,$ then moment problem (\ref{LPRY})--(\ref{LPRYZ}) may be simplified:

\begin{equation}
\left( x^{0},\psi _{j}\right) -e^{-\lambda _{j}t_{1}}\left( S_{Z}\left(
t_{1}\right) x_{Z}^{0},\psi _{Zj}\right) =  \tag{LPY}  \label{LPY}
\end{equation}%
\begin{equation}
=\int_{0}^{t_{1}}b_{1j}e^{-\lambda _{j}\tau }u\left( \tau \right) d\tau
,j\in I_{Y}\backslash I_{YZ}~,  \notag
\end{equation}%
\begin{equation}
\left( x^{0},\mathrm{\psi }_{1j}\right) +\left( x^{0},\mathrm{\psi }%
_{2j}\right) c_{j}t_{1}-e^{-\lambda _{j}t_{1}}\left( S_{Z}\left(
t_{1}\right) x_{Z}^{0},\mathrm{\psi }_{1Zj}\right) =  \tag{LPYZ}
\label{LPYZ}
\end{equation}%
\begin{equation}
=\int_{0}^{t_{1}}b_{1j}e^{-\lambda _{j}\tau }u\left( \tau \right) d\tau
,j\in I_{YZ},  \notag
\end{equation}%
or 
\begin{equation}
d_{j}=\int_{0}^{t_{1}}b_{1j}e^{-\lambda _{j}\tau }u\left( \tau \right) d\tau
,  \tag{LPYJ}  \label{LPYJ}
\end{equation}%
where

$d_{j}=\left\{ 
\begin{tabular}{l}
$\left( x^{0},\mathrm{\psi }_{1j}\right) -e^{-\lambda _{j}t_{1}}\left(
S_{Z}\left( t_{1}\right) x_{Z}^{0},\psi _{Zj}\right) ,~j\in I_{Y}\backslash
I_{YZ}$ \\ 
$\left( x^{0},\mathrm{\psi }_{1j}\right) +\left( x^{0},\mathrm{\psi }%
_{2j}\right) c_{j}t_{1}-e^{-\lambda _{j}t_{1}}\left( S_{Z}\left(
t_{1}\right) x_{Z}^{0},\mathrm{\psi }_{1Zj}\right) ,~j\in I_{YZ}.$%
\end{tabular}%
\right. $

Since series (\ref{xg_z(t1)}) converges for any $x_{Z}^{0}\in Z,$ we have $%
\left\{ d_{j}\in I_{Y}\right\} \in l^{2}.$

Hence the following corollaries are valid.

\begin{corollary}
\label{C4} If series (\ref{xg_z(t1)})~converges for any $x_{Z}^{0}\in Z,$
then system (\ref{CS1})--(\ref{P})) is partially null controllable on $%
[0,t_{1}]$~ if and only if the family of exponent%
\begin{equation}
\left\{ e^{-\lambda _{j}t}b_{1j},j\in I_{Y},t\in \left[ 0,t_{1}\right]
\right\}  \tag{-expg0}  \label{-expg0}
\end{equation}%
is strongly minimal.
\end{corollary}

\begin{corollary}
\label{C5} If the series (\ref{xg_z(t1)} converges for any $x_{Z}^{0}\in Z,$
then system (\ref{CS1})--(\ref{P}) is partially null controllable on $%
[0,t_{1}]$~ if and only if equation%
\begin{equation*}
\dot{y}\left( t\right) =A_{Y}y\left( t\right) +b_{Y}u\left( t\right)
\end{equation*}%
is null (completely) controllable on $[0,t_{1}].$
\end{corollary}

Theorems \ref{T2}--\ref{T3} and Corollaries \ref{C1}--\ref{C3} hold true for 
$\sigma _{YZ}\neq \varnothing $ as well.

\subsection{ Example 2}

%
%

The authors of \cite[system (12), page 4]{AmmarKhodja} have proved the
partial controllability conditions for the two-dimensional system 
\begin{equation}
\left\{ 
\begin{array}{c}
y_{t}^{\prime }=\triangle y+\alpha z+1_{\omega }u, \\ 
z_{t}^{\prime }=\triangle z%
\end{array}%
\right. ~\left( t,x\right) \in Q_{T}=\left[ 0,T\right] \times Q,  \tag{AMKho}
\label{AMKho}
\end{equation}%
\begin{equation*}
\begin{array}{c}
y\left( 0,x\right) =y_{0}\left( x\right) ,z\left( 0,x\right) =z_{0}\left(
x\right) ,x\in Q, \\ 
y\left( t,x\right) =z\left( t,x\right) =0,~t\in \left[ 0,T\right] ,x\in
\partial Q,%
\end{array}%
\end{equation*}%
where $Q$ is a bounded domain in $\mathbb{R}$ with a $C^{2}$-boundary $%
\partial Q,~\omega $ is a non-empty open subset of $Q$ and $T>0$,

$y_{0}\left( \cdot \right) ,z_{0}\left( \cdot \right) \in L_{2}(Q),u\left(
\cdot ,\cdot \right) \in L_{2}\left( Q_{T}\right) ,\alpha \left( \cdot
,\cdot \right) \in L_{\infty }\left( Q_{T},\mathcal{L}(\mathbb{R}\right) ).$

As it is noted in \cite{AmmarKhodja}, controllability of such partial
parabolic control systems are investigated in many fields such as chemistry,
physics or biology.

However, even if $u(t,x)\in \mathbb{R}$ for all $t\in \lbrack 0,T]$ and $%
x\in Q$, the control $u(\cdot ,\cdot )\in L(Q_{T})$ cannot be considered as
a single control force\footnote{%
it is infinite-dimensional control}. The given paper focuses on achieving
controllability using a single control force, so we will define $u(t,x)$ in
the form $\cdot )\in L(Q)$ is predetermined and $u(t),u(\cdot )\in L([0,T],%
\mathbb{R})$ represents a single control force.

The system being studied in this section is 
\begin{equation}
\left\{ 
\begin{array}{c}
y_{t}^{\prime }=\triangle y+\alpha z+\hat{b}u, \\ 
z_{t}^{\prime }=\triangle z,%
\end{array}%
\right. ~  \tag{exmp3}  \label{exmp3}
\end{equation}%
\begin{equation*}
\left( t,x\right) \in Q_{T}=\left[ 0,T\right] \times \left[ 0,\pi \right]
,~\alpha \in \mathbb{R}\backslash \{0\},\hat{b}\left( \cdot \right) \in
L_{2}\left( \left[ 0,\pi \right] ,\mathbb{R}\right) ,u\left( \cdot \right)
\in L_{2}\left[ 0,T\right],
\end{equation*}%
\begin{equation}
y\left( 0,x\right) =y_{0}\left( x\right) ,z\left( 0,x\right) =z_{0}\left(
x\right) ,x\in \lbrack 0,\pi ],  \tag{inc3}  \label{inc3}
\end{equation}%
\begin{equation}
y\left(t,0\right) =z\left( t,0\right) =0,~t\in \left[ 0,T\right],  \tag{buc3}
\label{buc3}
\end{equation}
\begin{equation*}
{y_0}\left( \cdot \right), {z_0}\left( \cdot \right) \in L_{2}\left( \left[
0,\pi \right]\right).
\end{equation*}
One can write system (\ref{exmp3})--(\ref{buc3}) in form (\ref{CS1})--(\ref%
{P}), where

$A=\left( 
\begin{array}{cc}
\triangle & \alpha \\ 
0 & \triangle%
\end{array}%
\right) ,b=\left( 
\begin{array}{c}
\hat{b} \\ 
0%
\end{array}%
\right) ,~P=\left( 
\begin{array}{cc}
1 & 0 \\ 
0 & 0%
\end{array}%
\right) .$

Here

$Y=\left\{ \left( 
\begin{array}{c}
y\left( x\right) \\ 
0%
\end{array}%
\right) ,y\left( \cdot \right) \in L_{2}\left[ 0,\pi \right] \right\} ,
Z=\left\{ \left( 
\begin{array}{c}
0 \\ 
z\left( x\right)%
\end{array}%
\right) ,z\left( \cdot \right) \in L_{2}\left[ 0,\pi \right] \right\} ,$

\begin{equation*}
A_{Y}=A_{Z}=\triangle
\end{equation*}%
with domains \cite[page 214, Example 4.6.6, page 241, Example 4.10.1]%
{Balakrishnan}

$D\left( A_{Y}\right) =D\left( A_{Z}\right) =\left\{ f,f^{\prime }~\mathrm{%
a.c.:}f^{\prime },f^{\prime \prime }\in L_{2}\left[ 0,\pi \right] \mathrm{%
~~and~}f\left( 0\right) =f\left( \pi \right) =0\right\}$ , \newline
$b_{Y}=\hat{b}, b_{Z}=0.$

Let $\lambda _{i}~$and $\varphi _{i},i\in \mathbb{N}~$ be eigenvalues and
eigenvectors of the operator $\Delta $ with boundary conditions (\ref{inc3}%
). Obviously, eigenvalues $\lambda_{n}=-n^{2},n=1,2,...,$ of $\triangle $
are eigenvalues of $A$, and to every eigenvalues $\lambda _{n}$ correspond
the eigenvectors $\phi _{1n}=\sqrt{\frac{2}{\pi }}\left( 
\begin{array}{c}
\sin nx \\ 
0%
\end{array}%
\right) ,n=1,2,...$ of the operator $A$, and the eigenvectors $\psi _{2n}=%
\sqrt{\frac{2}{\pi }}\left( 
\begin{array}{c}
0 \\ 
\sin nx%
\end{array}%
\right) $ of the operator $A^{\ast }.$ In this example both operator
equations (\ref{AYZG}) and (\ref{AYZG}*) are reduced to the equation%
\begin{eqnarray*}
y^{\prime \prime }+\alpha \sqrt{\frac{2}{\pi }}\sin nx &=&-n^{2}y+c_{n}\sqrt{%
\frac{2}{\pi }}\sin nx, \\
y\left( 0\right) &=&y\left( \pi \right) =0.
\end{eqnarray*}%
where according to (\ref{c}) we obtain 
\begin{equation*}
c_{n}=\alpha \int_{0}^{\pi }\frac{2}{\pi }\sin ^{2}nxdx=\alpha \neq 0,n\in 
\mathbb{N}.
\end{equation*}%
Hence equations $y^{\prime \prime }+\alpha \sqrt{\frac{2}{\pi }}\sin
nx=-n^{2}y+c_{n}\sqrt{\frac{2}{\pi }}\sin nx,~n\in \mathbb{N},$ are
equivalent to equations 
\begin{equation}
y^{\prime \prime }=-n^{2}y~,n\in \mathbb{N},y\left( 0\right) =\left( \pi
\right) =0,  \tag{geAA*}  \label{geAA*}
\end{equation}%
so the vector 
\begin{equation*}
\phi _{2n}=\left( 
\begin{array}{c}
\phi _{2Yn}\left( x\right) \\ 
\sqrt{\frac{2}{\pi }}\sin nx%
\end{array}%
\right) ,x\in \left[ 0,\pi \right]
\end{equation*}%
is a generalized eigenvector of $A,$ and the vector 
\begin{equation*}
\psi _{1n}=\left( 
\begin{array}{c}
\sqrt{\frac{2}{\pi }}\sin nx \\ 
\psi _{1Zn}\left( x\right)%
\end{array}%
\right) ,x\in \left[ 0,\pi \right]
\end{equation*}%
is a generalized eigenvector of the operator $A^{\ast },$ where $\phi
_{2Yn}\left( x\right) $ and $\psi _{1Zn}\left( x\right) $ are solutions of
equation (\ref{geAA*}). For any $n\in \mathbb{N}$ equation (\ref{geAA*}) has
the infinite number of solutions. Let's choose the solutions $\phi _{2Yn}=-%
\sqrt{\frac{2}{\pi }}\sin nx~$and $\psi _{1Zn}\left( x\right) =\sqrt{\frac{2%
}{\pi }}\sin nx,$ so the vector 
\begin{equation*}
\phi _{2n}=\left( 
\begin{array}{c}
-\sqrt{\frac{2}{\pi }}\sin nx \\ 
\sqrt{\frac{2}{\pi }}\sin nx%
\end{array}%
\right) ,x\in \left[ 0,\pi \right]
\end{equation*}%
is a generalized eigenvector of the operator $A,$ and the vector 
\begin{equation*}
\psi _{1n}=\left( 
\begin{array}{c}
\sqrt{\frac{2}{\pi }}\sin nx \\ 
\sqrt{\frac{2}{\pi }}\sin nx%
\end{array}%
\right) ,x\in \left[ 0,\pi \right]
\end{equation*}%
is a generalized eigenvector of the operator $A^{\ast }$, corresponding to
the eigenvalue $\lambda _{n}=-n^{2},n\in \mathbb{N}.$ 

One can prove ( see \cite{AmmarKhodjaChoulyDuprez} and Example 1), that the
sequences

\begin{equation}
\left\{ \phi _{1n},\phi _{2n},n\in \mathbb{N}\right\} ,  \tag{B3}  \label{B3}
\end{equation}
and 
\begin{equation}
\left\{ \psi _{1n},\psi _{2n},n\in \mathbb{N}\right\} ,  \tag{B3*}
\label{B3*}
\end{equation}
are Riesz basic in $X=L_{2}\left( \left[ 0,\pi \right] ,\mathbb{R}%
^{2}\right) .$ 
Hence system (\ref{exmp3})--(\ref{buc3}) satisfies to conditions of system (%
\ref{CS1})--(\ref{P})\footnote{%
See the bottom of Pages 2 and 3.}, so all theorems an corollaries proven in
previous sections, are applicable to Example 2.

Denote $b_{n}=\left( \hat{b},\sqrt{\frac{2}{\pi}}\mathrm{\sin }nx\right) = 
\sqrt{\frac{2}{\pi}}\dint_{0}^{\pi }\hat{b}\left( x\right) \sin nxdx,n\in 
\mathbb{N}.$

From (\ref{LPY})--(\ref{LPYZ}) it follows, that\footnote{%
Here $I_{Y}\backslash I_{YZ}$ is an empty set.} 
\begin{equation}
\left( x^{0},\mathrm{\psi }_{1n}\right) +\left( x^{0},\mathrm{\psi }%
_{2n}\right) \alpha t_{1}-e^{n^{2}t_{1}}\left( S_{Z}\left( t_{1}\right)
x_{Z}^{0},\mathrm{\psi }_{1Zn}\right) =\int_{0}^{t_{1}}b_{n}e^{n^{2}\tau
}u\left( \tau \right) d\tau ,n\in \mathbb{N},  \tag{LMPYZ}  \label{LMPYZ}
\end{equation}
where $\mathrm{\psi }_{1Zn}=\sqrt{\frac{2}{\pi }}\sin nx, n\in \mathbb{N}$.

Now we will check the convergence of series (\ref{xg_z(t1)}).

We have

$S_{Z}\left( t_{1}\right) x_{Z}^{0}=\sum_{m=1}^{\infty
}x^0_{Zm}e^{-m^{2}t_{1}} \sqrt{\frac{2}{\pi }}\sin mxdx,0\leq x \leq \pi ,$
where

$x_{Zm}^{0}=\int_{0}^{\pi }x_{Z}^{0}\left( \theta \right) \sqrt{\frac{2}{\pi 
}}\sin m\theta d\theta $,


$(S_{Z}\left( t_{1}\right) x_{Z}^{0},\mathrm{\psi }_{1Zn})=\sum_{m=1}^{%
\infty }e^{-m^{2}t_{1}}x_{Zm}^{0}\dint_{0}^{\pi }\left( \sqrt{\frac{2}{\pi }}%
\sin mx\sqrt{\frac{2}{\pi }}\sin nx\right) dx=x_{Zn}^{0}e^{-n^{2}t_{1}},$

$\sum_{n=1}^{\infty }e^{2n^{2}t_{1}}\left\vert \left( S_{Z}\left(
t_{1}\right) x_{Z}^{0},\psi _{Zn}\right) \right\vert ^{2}=\sum_{n=1}^{\infty
}e^{2n^{2}t_{1}}e^{-2n^{2}t_{1}}\left\vert x_{Zn}^{0}\right\vert
^{2}=\left\Vert x_{Z}^{0}\right\Vert ^{2}<+\infty .\forall t_{1}>0.$

Hence from Corollary \ref{C4} it follows, that system (\ref{exmp3})--(\ref%
{buc3}) is partially null controllable on $\left[ 0,t_{1}\right] $ if and
only if the family

\begin{equation}
\left\{ e^{n^{2}t}b_{n},n\in \mathbb{N},t\in \left[ 0,t_{1}\right] \right\} 
\tag{-expn2}  \label{-expn2}
\end{equation}%
is strongly minimal.

The eigenvalues $\lambda _{n}=-n^{2}, n\in \mathbb{N}$ satisfy he sufficient
conditions for strong minimality of family (\ref{-expn2}) , obtained in
Theorem \ref{T3} with $\rho =1$. 
In accordance with Theorem \ref{T3} and Corollary \ref{C3} system (\ref%
{exmp3})--(\ref{buc3}) is partially null controllable on $\left[0,t_{1}%
\right] $ for any $t_{1}>\lim_{n\rightarrow \infty } \frac{2\ln \left\vert
b_{n}\right\vert }{-n^{2}}$.

\section{Conclusion}

%

The partial exact null-controllability conditions for abstract coupled
control equations (\ref{e1})--(\ref{e3}) governed by a single distributed
control $u(t)$ have been obtained using the moment problem approach. These
conditions strengthen the partial controllability conditions for one of
coupled parabolic systems obtained in Ammar Khodja et al. (2016). This is
because, as demonstrated in Example 2, even when $u(t,x)\in \mathbb{R}$ for
all $t\in [0,T]$ and $x\in Q$, the control $u(\cdot,\cdot)\in L(Q)$ cannot
be treated as a single control force in the conventional sense\footnote{%
it is infinite-dimensional control}.

The controllability results to date have focused on control problems for
coupled PDEs, mostly of the same structure. The majority of results have
been obtained for coupled parabolic equations, and a few have been done for
the controllability of coupled wave equations. The abstract approach
developed here allows us to investigate various controllability problems,
including partial controllability, for coupled systems containing equations
of different structures.

All examples are simple; they have been chosen only for illustrative
purposes.


\end{document}